\documentclass[reqno]{amsart}
\usepackage{amsfonts}
\usepackage{dsfont}
\usepackage{amssymb,amscd,amsthm, verbatim,color, fancyhdr, mathrsfs}
\usepackage{amsmath}
\usepackage{comment}
\usepackage{graphicx}
\usepackage{turnstile}
\usepackage{cite}
\usepackage{tikz-cd}
\usepackage{thmtools}
\usepackage{mathtools}
\usepackage{thm-restate}
\usepackage{tikz-cd}
\usepackage{float}

\usepackage[margin=1.25in]{geometry}
\usepackage[font=small,labelfont=bf]{caption}
\usepackage{stmaryrd}
\usepackage{pict2e}
\usepackage[english]{babel}
\usepackage[hidelinks]{hyperref}
\hypersetup{colorlinks}
\definecolor{darkred}{rgb}{0.5,0,0}
\definecolor{darkgreen}{rgb}{0,0.5,0}
\definecolor{darkblue}{rgb}{0,0.3,0.8}
\hypersetup{colorlinks, linkcolor=black, filecolor=darkgreen, urlcolor=darkred, citecolor=darkblue}
\usepackage{cleveref}
\usepackage{kpfonts}
\usepackage[titletoc]{appendix}
\usepackage[numbers,square]{natbib} 
\usepackage{xcolor}
\usepackage[all,cmtip]{xy}
\usepackage{enumitem}
\usepackage{cases}
\usepackage{setspace}

\numberwithin{equation}{section}
\begingroup\lccode`~=`& \lowercase{\endgroup
	\newcommand{\spmat}[1]{%
		\left(
		\let~=&
		\begin{smallmatrix}#1\end{smallmatrix}
		\right)
}}

\makeatletter
\DeclareRobustCommand{\intprod}{
	\mathbin{\mathpalette\intprod@{(0.1,0)(0.9,0)(0.9,0.8)}}}
\newcommand{\intprod@}[2]{
	\begingroup
	\setbox0=\hbox{$\m@th#1+$}%
	\unitlength=\wd0
	\begin{picture}(1,1)
		\roundcap
		\linethickness{0.1\unitlength}
		\polyline#2%
	\end{picture}%
	\endgroup
}
\makeatother

\newtheorem*{exr*}{Exercise}
\newtheorem*{thm*}{Theorem}
\newtheorem{thm}{Theorem}[section]
\newtheorem{prop}[thm]{Proposition}
\newtheorem{lem}[thm]{Lemma}
\newtheorem{rem}[thm]{Remark}

\newtheorem*{Conjecture}{Conjecture}

\newtheorem*{Construction}{Construction}
\newtheorem{cor}[thm]{Corollary}

\newtheorem{defn}[thm]{Definition}
\newtheorem*{assump}{Assumption}

\theoremstyle{definition}
\newtheorem{example}[thm]{Example}

\newcommand{\xrightarrowdbl}[2][]{%
	\xrightarrow[#1]{#2}\mathrel{\mkern-14mu}\rightarrow
}

\newcommand{\Z}{\mathbb{Z}}
\newcommand{\R}{\mathbb{R}}

\newcommand{\Q}{\mathbb{Q}}
\newcommand{\C}{\mathbb{C}}
\newcommand{\bS}{\mathbb{S}}

\newcommand{\bb}{\boldsymbol{b}}
\newcommand{\bv}{\mathbf{v}}
\newcommand{\bw}{\mathbf{w}}

\newcommand{\CF}{\mathrm{CF}}
\newcommand{\HF}{\mathrm{HF}}

\newcommand{\bL}{\mathbb{L}}

\newcommand{\CP}{\mathbb{C}P}

\newcommand{\cF}{\mathscr{F}}

\newcommand{\cS}{\mathcal{S}}

\newcommand{\cH}{\mathcal{H}}

\newcommand{\bF}{\mathbf{F}}

\newcommand{\fr}{\textrm{fr}}

\newcommand{\Hom}{\mathrm{Hom}}

\newcommand{\End}{\mathrm{End}}
\newcommand{\Fuk}{\mathrm{Fuk}}
\newcommand{\Ker}{\mathrm{Ker}\,}
\renewcommand{\Im}{\mathrm{Im}\,}
\newcommand{\GL}{\mathrm{GL}}

\newcommand{\cE}{\mathcal{E}}

\newcommand{\tr}{\mathrm{tr} \,}

\newcommand{\one}{\mathbf{1}}

\newcommand{\bP}{\mathbb{P}}

\newcommand{\bM}{\mathcal{M}}
\newcommand{\bMC}{\mathcal{MC}}
\newcommand{\cA}{\mathcal{A}}

\newcommand{\cD}{\mathcal{D}}

\newcommand{\cL}{\mathcal{L}}
\newcommand{\cM}{\mathcal{M}}

\newcommand{\MF}{\mathrm{MF}}

\newcommand{\Id}{\mathrm{Id}}

\newcommand{\cG}{\mathscr{G}}

\graphicspath{ {./pics/} }

\title{Mirror construction of Hecke correspondence between Nakajima quiver varieties}

\author{Siu-Cheong Lau and Ju Tan}

\date{\today}
\begin{document}
	\begin{abstract}
	Nakajima constructed geometric representations of a deformed Kac-Moody Lie algebra using Hecke correspondences between quiver varieties. In this paper, we show that Hecke correspondences, which are holomorphic Lagrangians in products of Nakajima quiver varieties, can be obtained by applying the localized mirror construction to the morphism spaces between families of framed Lagrangian branes supported on the core of a plumbing of two-spheres. 
	Moreover, for a non-ADE quiver, we show that the localized mirror functor is fully-faithful.
	\end{abstract}
	\maketitle 
	{ \hypersetup{hidelinks} \tableofcontents }
\newpage
\section{Introduction}


Nakajima quiver varieties for affine ADE quivers were originally introduced as moduli spaces of anti–self-dual (ASD) instantons on ALE spaces and as moduli of framed torsion-free sheaves over their compactifications \cite{KN90,Nak94,Nak07}. 
Subsequently, these holomorphic symplectic varieties for general framed double quivers were found to play a central role in geometric representation theory, providing a powerful framework for realizing the representations of deformed Kac–Moody and quantum affine algebras via Hecke correspondences \cite{Nak98,Nak01}.



On the other hand, for a Lagrangian immersion $L$ in a symplectic manifold $M$, its noncommutative deformation space (which serves as a mirror of $M$ localized to $L$) was formulated as quiver algebras in \cite{CHL17,CHL21}. It is natural to ask whether Nakajima quiver varieties lie in this framework, motivated from SYZ mirror symmetry and family Floer theory \cite{SYZ96,Fuk02,Tu,Abouzaid17}. 

Our previous work \cite{HLT24} provided an affirmative answer to this question: 
we developed the notion of a framed Lagrangian brane, constructed explicit examples in dimension two and computed their stable deformation spaces. By analyzing the Floer theoretical obstructions, we showed that these deformation spaces are isomorphic to Nakajima quiver varieties.\footnote{More precisely, we need to work over the Novikov ring $\Lambda_{\geq 0}$ with formal parameter $T$ for convergence in Floer theory. On the other hand, we will identify with objects of $\cM(\bv,\bw)$ that only have polynomial dependence in $T$, so that we can set $T=e^{-1}$ to get back to the complex field $\C$.}

\begin{thm}[\cite{HLT24}]\label{thm: MC}
	Let $D$ be a non-directed graph and $Q$ the corresponding double quiver. Let $\bM(\bv,\bw)$ be the Nakajima quiver variety associated to the framed quiver $Q^\fr$, where $\bv$ and $\bw$ are the dimension vectors for $Q$ and the framing respectively.
	
	There exists a framed Lagrangian brane $(\bL^\fr,\cE)$ in a symplectic manifold $M$, where $\bL^\fr$ is a framed Lagrangian immersion and $\cE$ consists flat bundles of rank $\bv$ and $\bw$ over its compact and framing components respectively, such that its Maurer-Cartan deformation space $\bMC(\bL^{\fr},\cE)$ is isomorphic to the Nakajima quiver variety:
	$$ \bMC(\bL^{\fr},\cE) \cong \bM(\bv,\bw). $$
	In above, $\bMC(\bL^{\fr},\cE)$ is defined as the GIT quotient of the solution space of the Maurer-Cartan equation for deformations of $(\bL^{\fr},\cE)$ by the gauge group associated to $(\bL^{\fr},\cE)$, with the stability condition $\zeta$ adapted to the framing which has GIT character $0$ on the framing component and $1$ on the compact components. The Lagrangian immersion $\bL^\fr \subset M$ is constructed by plumbing of 
	$T^*\bS^2$ according to the graph $D$.
\end{thm}



The above result serves as the first step to categorify geometric representations of Kac-Moody algebras via Fukaya category, namely, it realizes the moduli spaces $\bMC(\bL^{\fr},\cE)$ of framed objects in the category as holomorphic symplectic manifolds. 

From this viewpoint, a natural next goal is to understand how the representation-theoretic structures encoded in Nakajima’s theory arise on the symplectic side. More precisely, in Nakajima’s theory, the module actions of Kac-Moody algebras are realized by Fourier-Mukai transformations between the derived category of coherent sheaves supported over holomorphic Lagrangians in $\bMC(\bL^{\fr},\cE)$. To carry out an analogous categorification on the symplectic side, a key missing ingredient is the symplectic counterparts of the correspondences that drive Nakajima’s construction. This leads to the following central question:

\begin{center}
	\emph{What are the structures corresponding to the Hecke correspondences between Nakajima quiver varieties in Lagrangian Floer theory?}
\end{center}

One of the main results of this paper is the following symplectic construction of the Hecke correspondence, which induces maps on the Borel–Moore homology of quiver varieties. These correspondences form the essential geometric input to construct actions of Kac-Moody algebras on the direct sum of Borel–Moore homology of quiver varieties. The theorem below realizes these correspondences by morphisms between framed Lagrangian branes.

\begin{thm}[Theorem \ref{thm:H}]
	Assume $D$ has no edge loops. Consider two branes $(\bL^\fr,\cE_j)$ for $j=1,2$ where $\cE_j$ are trivial bundles over $\bL^\fr$ of rank $(\bv^j,\bw)$ for $\bv^2 = \bv^1 + \mathbf{e}^k$ (meaning that $\cE_2$ has one higher rank than $\cE_1$ at the $k$-th component).  Then the support of the sheaf $\mathcal{H}$ over $\bMC(\bL^{\fr},\cE_2) \times \bMC(\bL^{\fr},\cE_1)$ formed by the Floer cohomology $\HF^2((\bL^\fr,\cE_2),(\bL^\fr,\cE_1))$ is a holomorphic Lagrangian subvariety 
	$$\mathfrak{P}_k \subset \bMC(\bL^{\fr},\cE_2) \times \bMC(\bL^{\fr},\cE_1),$$
	which is isomorphic with the Hecke correspondence in $\cM(\bv^2,\bw) \times \cM(\bv^1,\bw)$.
	Moreover, 
	$$\mathcal{H} = \iota_* \mathcal{L}_{\mathfrak{P}_k} $$ 
	where $\iota: \mathfrak{P}_k \to \bMC(\bL^{\fr},\cE_2) \times \bMC(\bL^{\fr},\cE_1)$ is the inclusion map and $\mathcal{L}_{\mathfrak{P}_k}$ is a line bundle over $\mathfrak{P}_k$.
\end{thm}

Let us briefly explain the proof. By definition, the Hecke correspondence is a subvariety in the product space $\cM(\bv^1,\bw) \times \cM(\bv^2,\bw)$ that parametrizes pairs of representations $(V^1,V^2)$ in the product together with an injective morphism $V^1 \xhookrightarrow[]{} V^2$. We consider the Floer complex $\CF^*((\bL^\fr,\cE_2),(\bL^\fr,\cE_1))$, which was observed to be isomorphic with Nakajima's generalized monadic complex in \cite{HLT24}. For fixed $b_1$ and $b_2$, $\HF^2((\bL^\fr,\cE_2,b_2),(\bL^\fr,\cE_1,b_1))$ 
is dual to 
$\HF^0((\bL^\fr,\cE_1,b_1),(\bL^\fr,\cE_2,b_2))$.
Besides, we will show that 
$\HF^0((\bL^\fr,\cE_1,b_1),(\bL^\fr,\cE_2,b_2))$ 
is isomorphic to the morphism space of the corresponding framed quiver representations. 
A morphism in $\HF^0$ is nonzero precisely when the morphism 
between the framing vertices is nontrivial; 
furthermore, by the stability condition $\zeta$ stated in Theorem \ref{thm: MC}, a nonzero morphism is injective. Hence, after an appropriate change of coordinates, the nonzero locus of $\HF^0$ (equivalently its dual $\HF^2$) parametrizes the space of injective morphisms between framed representations—that is, the Hecke correspondence.

Next, we consider sheaves that are supported over holomorphic Lagrangians in $\bMC(\bL^{\fr},\cE)$. As a GIT quotient at the stability condition $\zeta$, it has a projective morphism $$\pi:\bMC(\bL^{\fr},\cE)\to \bMC_0(\bL^{\fr},\cE)$$
where $\bMC_0(\bL^{\fr},\cE)$ is the GIT quotient with the trivial stability condition. 
By Theorem \ref{thm: MC} and the result of Nakajima, the preimage of trivial representation $\pi^{-1}([0])$ is a holomorphic Lagrangian subvariety. We verify that $\pi^{-1}([0])$ can be produced by the mirror functor applied to the framing Lagrangian $F$ in the ADE case.

\begin{thm}[Theorem \ref{thm: lag}, Corollary \ref{cor: hlag}]
	Let $D$ be a graph, and $F$ be the framing component of $\bL^\fr$. The support of the sheaf over $\bMC(\bL^{\fr},\cE)$ whose fibers are $\HF^2((\bL^\fr,\cE,b),F)$ forms a subvariety in $\bMC(\bL^{\fr},\cE)$ defined by $i=0$. 
	
	Moreover, if $D$ is an ADE Dynkin diagram, $$\HF^2((\bL^\fr,\cE),F) \cong \iota_* \mathcal{O}_{L(\bv)} $$ 
	where $L(\bv) \subset \bMC(\bL^\fr,\cE)$ is the holomorphic Lagrangian subvariety $\pi^{-1}([0])$ and $\iota: L(\bv) \to \bMC(\bL^{\fr},\cE)$ is the inclusion map.
\end{thm}

The assumption that $D$ is an ADE Dynkin diagram is necessary to ensure that the resulting subvarieties are holomorphic Lagrangians. For general graphs, the subvariety defined by $i=0$ may have higher dimension, see Example \ref{ex: cexample}.

Moreover, this work fits into the general framework of Homological Mirror Symmetry conjecture proposed by Kontsevich \cite{Kon95}, which predicts an equivalence between the Fukaya category of a symplectic manifold and the derived category of coherent sheaves on its mirror.

In our setting, by applying the localized mirror functor, we establish a local homological mirror symmetry equivalence for the universal brane $(\bL,\bb)$.

\begin{thm}[Theorem \ref{thm:qiso}]
	Let $M$ be the plumbing of cotangent bundles $T^*\bS^2$ according to a non-Dynkin diagram $D$, and $\bL \subset M$ be the zero section. Then the noncommutative deformation space $\cA_\bL$ is a 2-Calabi-Yau algebra, and the localized mirror functor $\cF^{(\bL,\bb)}(\bL,\bb)$ gives a self-dual projective resolution of $\cA_\bL$ as an $\cA_\bL$-bimodule. 
	
	Moreover, the localized mirror functor $$\cF^{(\bL,\bb)}:\Fuk^{sub}(M) \to \mathrm{Perf}(\cA_\bL)$$ induces a quasi-equivalence, where $\Fuk^{sub}(M)$ is the Fukaya subcategory split generated by $(\bL,\bb)$ and and $\mathrm{Perf}(\cA_\bL)$ is the dg category of perfect $\cA_\bL$-modules.
\end{thm}

This quasi-equivalence can be viewed as a local realization of Kontsevich’s Homological Mirror Symmetry conjecture, in which the Fukaya subcategory generated by the universal brane corresponds to a local chart of the mirror, represented by the noncommutative algebra $\cA_\bL$.

In addition, we also obtain the following three-dimensional analogue.

\begin{prop}[Proposition \ref{prop: 3d}]
	Let $\bL$ be a compact, relatively spin, oriented Lagrangian immersion of dimension 3 and $\cA_\bL$ be its noncommutative deformation space. If $\cA_\bL$ is a 3-Calabi-Yau algebra, then the localized mirror functor $\cF^{(\bL,\bb)}(\bL,\bb)$ gives a self-dual projective resolution of $\cA_\bL$ as an $\cA_\bL$-bimodule. 
	
	Moreover, the localized mirror functor $$\cF^{(\bL,\bb)}:\Fuk^{sub}(M) \to \mathrm{Perf}(\cA_\bL)$$ induces a quasi-equivalence.
\end{prop}

Moreover, in these two cases, the noncommutative deformation space $\cA_\bL$ is Calabi-Yau if and only if the localized mirror functor $\cF^{(\bL,\bb)}(\bL,\bb)$ provides a self-dual projective $\cA_\bL$-bimodule resolution of $\cA_\bL$.

The paper is organized as follows.
Section \ref{sec: review} reviews the necessary background on convolution algebras, and Nakajima quiver varieties.
Section \ref{sec: gen} develops the general Floer-theoretic framework, introducing the localized mirror construction for families of higher-rank Lagrangian branes.
Section \ref{sec: H} gives the Floer-theoretic realization of Hecke correspondences and the construction of holomorphic Lagrangians in Nakajima quiver varieties.
Section \ref{sec: CY} establishes the full-faithfulness of the localized mirror functor and explains how this construction fits into local Homological Mirror Symmetry.

\subsection*{Acknowledgments}
We thank Hansol Hong for very valuable discussions on algebraic resolutions, Koszul duality and other related topics during his stay at Boston University in sabbatical. We are grateful to Naichung Conan Leung and Kwokwai Chan for enlightening discussions and invitations to The Chinese University of Hong Kong, and Helge Ruddat for the Miami mirror symmetry workshop in November 2024. The first author expresses his gratitude to Hiroshi Iritani and Kaoru Ono for the Kyoto workshop in December 2024 and useful discussions on disc moduli and Kuranishi structures. He thanks Cheol-Hyun Cho and Yoosik Kim for the invitation to Seoul National University in summer 2024 and the very interesting discussions there. He is also grateful to Banff International Research Station for the support of the mirror symmetry workshop in October 2025. The second author expresses his gratitude to Sheel Ganatra, whose questions about the mirror–symmetric correspondence of stability conditions motivated the present consideration of local homological mirror symmetry. He also thanks Ki Fung Chan, Eddie Lam and Yan Lung Leon Li for many stimulating discussions on Hecke correspondences.

\section{Quick review on Nakajima quiver varieties}\label{sec: review}

\subsection{Convolution Algebra} Let us first recall the convolution product for the Borel-Moore homology. More discussions about the convolution product can be found in \cite{FM81,CG97}.

Let $X$ be a topological space that can be embedded as a closed subspace of a Euclidean space $\R^m$. We define a version of Borel-Moore homology groups (see the Appendix B of \cite{Ful97} for the fundamental properties), denoted by $H_*(X)$, by the formula
$$H_i(X):=H^{m-i}(\R^m, \R^m \setminus X),$$ where the right-hand side is the relative singular cohomology group.

The definition is independent of the choice of an embedding: if $X$ is embedded as a closed subspace of an oriented smooth manifold $M$, there is a canonical isomorphism $$H_i(X) \cong H^{n-i}(M,M \setminus X),$$ where $n$ is the dimension of $M$.
If $f:X \to Y$ is a proper map, there is a pushforward homomorphism: $$f_*: H_i(X) \to H_i(Y).$$ 
If $\iota:U \to Y$ is an open embedding, there's a pullback homomorphism: $$\iota^*:H_i(Y) \to H_i(U).$$ 
If $X$ and $Y$ are closed subsets of an $n$-dimensional oriented manifold $M$, we have the cup product in the relative cohomology group $$\cup: H^{n-i}(M,M \setminus X) \otimes H^{n-j}(M,M \setminus Y) \to H^{2n-i-j}(M,M \setminus (X\cap Y)).$$ By the canonical isomorphism introduced before, it induces a cap product in the Borel-Moore homology group:
$$\cap: H_i(X) \otimes H_j(Y) \to H_{i+j-n}(X \cap Y).$$ Note that this product depends on the ambient space $M$.


Let $M^1,M^2$ and $M^3$ be oriented smooth manifolds and $p_{ij}:M^1 \times M^2 \times M^3 \to M^i \times M^j$ the natural projection. Let $Z \subset M^1 \times M^2$ and $Z' \subset M^2 \times M^3$ be closed subsets. By the cap product in $M^1 \times M^2 \times M^3,$ we have $$\cap: H_{i+d_3}(p_{12}^{-1}Z) \otimes H_{j+d_1}(p_{23}^{-1}Z') \to H_{i+j-d_2}(p_{12}^{-1}Z \cap p_{23}^{-1}Z'),$$ where $d_i$ is the dimension of $M^i$.

Assume the map $p_{13}: p_{12}^{-1}Z \cap p_{23}^{-1}Z' 
\to M^1 \times M^3$ is proper. Its image is denoted by $Z \circ Z'.$ We define the convolution by 
\begin{equation}
	*:H_i(Z) \otimes H_j(Z') \to H_{i+j-d_2}(Z \circ Z'), \quad c *c':=(p_{13})_*(p_{12}^*c \cap p_{23}^*c'), 
\end{equation} where $p_{12}^* c= c \times [M^3]$ and so on. This makes sense for disconnected manifolds as well.

\subsection{Nakajima Quiver Varieties}
In this subsection, we fix notaions for Nakajima quiver varieties. Let $D:=(I,E)$ be a finite graph, where $I$ is the set of vertices and $E$ the set of edges. Let $\textbf{A}$ be the adjacency matrix of the graph. Then $\textbf{C=2I-A}$ is a (symmetric) Cartan matrix.

Let $H$ be the set of pairs consisting of an edge together with its orientation. For $a \in H$, we denote $h(a)$ (resp. $t(a)$) the head (resp. tail) vertex of $a$. For $a \in H$, we denote $\Bar{a}$ the same edge as $a$ with the reverse orientation. An orientation $\Omega$ of the graph is a subset $\Omega \subset H$ such that $\Bar{\Omega} \cup \Omega=H, \Bar{\Omega} \cap \Omega= \emptyset.$ The orientation defines a function $\epsilon:H \to \{ \pm 1\}$ given by $\epsilon(a)=1$ if $a \in \Omega$ and $=-1$ if $a \in \Bar{\Omega}.$

Let $V= \oplus_{i \in I} V_i$ be an $I$-graded vector space. We define its dimension vector by 
$$\mathbf{v} = \dim V:= (\dim V_i)_{i \in I} \in \Z^I_{\geq 0}.$$ 
If $V^1, V^2$ are $I$-graded vector spaces, we introduce vector spaces $$L(V^1, V^2):=\oplus_{i \in I}\Hom(V^1_i, V^2_i)$$ which is the space of linear maps among the vector spaces over the same vertices. We also introduce $$E(V^1, V^2):=\oplus_{a \in H} \Hom(V^1_{t(a)},V^2_{h(a)})$$ which is the representation of the double quiver $Q$, see Definition below.

\begin{defn}
	The double quiver $Q$ of a graph $D:=(I,E)$ is the the quiver with the same vertices of $D$ and with the set of oriented edges $H$.
\end{defn} 

For $B=(B_a) \in E(V^1, V^2), C=(C_a) \in E(V^2,V^3),$ we define a multiplication of $B$ and $C$ by \begin{equation} \label{eq: mult}
	CB=\left(\sum_{h(a)=i} C_a B_{\Bar{a}}\right)_i \in L(V^1, V^3).
\end{equation}

Multiplications $Ba,Ca$ of $a \in L(V^1, V^2), B \in L(V^2,V^3), C\in E(V^2,V^3)$ are defined in similar manner. If  $a \in L(V^1,V^1)$, its trace $tr(a)$ is understood as $\sum_k tr(a_k).$

Let $V$ and $W$ be $I$-graded vector spaces. We define $$M(V,W):=E(V,V) \oplus L(V,W) \oplus L(W,V).$$
$M(V,W)$ is the representation space of the framed quiver $Q^\fr$, which is the double of $D$ with framing at each vertex. $W$ is called the framing whose dimension vector is denoted by $\bw$. The representation space of a double quiver is isomorphic to the cotangent space of the quiver representation space before doubling. Sometimes, we will write $M(\bv,\bw)$ instead of $M(V,W)$ to emphasize the dimension of the representations.

An element in $M(V,W)$ will be denoted by $(B, a, b)$. This space has a holomorphic symplectic form given by $$\omega((B,a,b),(B',a',b')):=tr(\epsilon BB') + tr(ab'-a'b),$$
where $\epsilon B$ is an element of $E(V,V)$ defined by $(\epsilon B)_a= \epsilon(a)B_a.$

Let $G:=G_\bv$ be the Lie group $\prod_i GL(V_i)$. It acts on $M(V,W)$ via $$g \cdot (B,a,b) \mapsto (gBg^{-1}, ag^{-1}, gb),$$
which preserves the symplectic form. The moment map is given by $$\mu(B,a,b)=\epsilon BB+ab \in L(V,V),$$ where the dual of the Lie algebra of $G$ is identified with $L(V,V)$ via the trace and the multiplication is defined as in Equation \ref{eq: mult}. 

Let $\zeta_\C=(\zeta_{\C,i}) \in \C^I $. We define a corresponding element in the center of $\mathrm{Lie} (G)$ by $\oplus_i \zeta_{\C,i} id_{V_i},$ where we delete the summand corresponding to $i$ if $V_i=0.$ Let $\mu^{-1}(\zeta_\C)$ be an affine algebraic variety (not necessarily irreducible) defined as the zero set of $\mu- \zeta_\C$. The group $G$ acts on $\mu^{-1}(\zeta_\C).$

We now define the stability conditions and Nakajima quiver varieties.

For $\zeta_\R= (\zeta_{\R,i})_{i \in I} \in \R^I,$ let $\zeta_\R \cdot \bv:= \sum_i \zeta_{\R,i} \bv_i.$

\begin{defn}\label{def: stab}
	A point $(B,a,b) \in M(V,W)$ is $\zeta_\R$-semistable if the following two conditions are satisfied:
	\begin{enumerate}
		\item If an $I$-graded subspace $S$ of $V$ is contained in $\Ker\, b$ and $B$-invariant, then $\zeta_\R \cdot \dim S \leq 0.$
		\item If an $I$-graded subspace $T$ of $V$ contains in $\Im\, a$ and $B$-invariant, then $\zeta_\R \cdot \dim T \leq \zeta_\R \cdot \dim V.$
	\end{enumerate}
	We say $(B,a,b)$ is $\zeta_\R$-stable if the strict inequalities hold in $1,2$ unless $S=0$, $T=V$ respectively.
\end{defn}
\begin{rem}\label{rem:stability}
	If  $\zeta_{\R,i}>0$ for all $i$, the condition $(2)$ is superfluous and the condition $(1)$ is that there is no nonzero $B$-invariant $I$-graded subspaces $S$ contained in $\Ker \, b$.
\end{rem}

Let $H^s_{\zeta_\R,\zeta_\C}$ (resp. $H^{ss}_{\zeta_\R,\zeta_\C}$) be the set of $\zeta_\R$-stable (resp. $\zeta_\R$-semistable) points in $\mu^{-1}(\zeta_\C).$ We say that two $\zeta_\R$-semistable points $(B,a,b), (B',a',b')$ are $S$-equivalent when the closures of $G_\bv$-orbits intersect in $H^{ss}_{\zeta_\R,\zeta_\C}$. We denote the pair $(\zeta_\R,\zeta_\C)$ by $\zeta$ for brevity. 
\begin{defn}
	The Nakajima quiver variety is defined to be $$\mathcal{M}_\zeta:=\mathcal{M}_\zeta(\bv,\bw):=H^{ss}_{\zeta_\R,\zeta_\C}/ \sim,$$ where $\sim$ denotes the $S-$equivalence relation. If the stability condition is clear in the context, it will be denoted by $\bM(\bv,\bw)$. The regular part is defined by
	$$\mathcal{M}^{reg}_\zeta:=\mathcal{M}^{reg}_\zeta(\bv,\bw):=H^{s}_{\zeta_\R,\zeta_\C}/ G_\bv.$$
\end{defn}

\begin{rem}\label{rem: tau}
	For a generic stability condition, $H^{s}_{\zeta_\R,\zeta_\C}$ is a principal $G_\bv$-bundle over $\mathcal{M}_\zeta(\bv,\bw)$. Since $V_i$ is a $G_\bv$ representation, one can associate the tautological vector bundles $\mathcal{V}_i:=H^{s}_{\zeta_\R,\zeta_\C} \times_{G_\bv} V_i$ over $\mathcal{M}_\zeta(\bv,\bw)$ for each vertex $i$. As for vector spaces, one can define the holomorphic vector bundles $L(\mathcal{V}^1,\mathcal{V}^2)$ and $E(\mathcal{V}^1,\mathcal{V}^2).$ We will also use this notations to describe the Lagrangian Floer complexes.
\end{rem}

Next, we briefly recall the definition of the Hecke correspondence following Nakajima \cite[§5]{Nak98}.
\begin{defn}
	Fix a vertex $k\in I$ and set $\bv^2=\bv^1+\mathbf{e}^k$.
	The \emph{Hecke correspondence}
	\[
	\mathfrak{P}_k(\bv^2,\bw)
	\subset \cM(\bv^1,\bw)\times \cM(\bv^2,\bw)
	\]
	is defined to be the subvariety consisting of pairs of framed representations
	\[
	\big([B_1,i_1,j_1],[B_2,i_2,j_2]\big)
	\]
	for which there exists an injective homomorphism 
	$\xi:V^1\hookrightarrow V^2$ satisfying
	\[
	B_2\circ\xi=\xi\circ B_1,\qquad i_2=\xi\circ i_1,\qquad j_1=j_2\circ\xi.
	\]
\end{defn}
Equivalently, $\mathfrak{P}_k(\bv^2,\bw)$ parameterizes $B_2$–invariant subspaces 
$S=\operatorname{Im}(\xi)\subset V^2$ containing $\operatorname{Im}(i_2)$ 
with $\dim(V^2/S)=\mathbf{e}^k$. 
Moreover, Nakajima \cite{Nak98} proved that $\mathfrak{P}_k(\bv^2,\bw)$ is a holomorphic Lagrangian subvariety. Its fundamental cycle serves as the geometric kernel of the creation and annihilation operators on the Borel-Moore homology of quiver varieties. This produces an 
integrable highest-weight representation of the deformed Kac–Moody algebra 
\cite[Thm.~10.3]{Nak98}.

\section{Representation theory from moduli spaces of Lagrangian branes}\label{sec: gen}
\subsection{Operators associated to moduli of branes supported on a Lagrangian immersion}
\label{sec:op_Fuk}
Let $X$ be a symplectic manifold. We assume that $X$ is either compact or convex non-compact. 
\begin{defn}
	A Lagrangian brane in $X$ is a pair $(\bL,\cE)$, where $\bL \subset X$ is a relatively-spin oriented graded Lagrangian immersion and $\cE \to \hat{\bL}$ is a flat vector bundle over the normalization $\hat{\bL}$ of $\bL$.
\end{defn} Under suitable choices of perturbation data, $(\bL,\cE)$ is associated with an $A_\infty$ algebra 
\begin{equation} \label{eq:CF}
	\CF((\bL,\cE),(\bL,\cE)) = C^*(\hat{\bL},\End(\cE)) \oplus \bigoplus_{a} \left(C^*(S_a,\Hom(\cE|_{\iota_-(S_a)},\cE|_{\iota_+(S_a)})) \oplus C^*(S_a,\Hom(\cE|_{\iota_+(S_a)},\cE|_{\iota_-(S_a)}))\right)
\end{equation}
by Lagrangian Floer theory \cite{FOOO09,AJ10,FOOO-can}, where $a$ indexes the clean intersections $S_a$ of $\bL$ and $\iota_\pm:S_a \to \hat{\bL}$ is the embedding to its normalization.  
We will denote the component contributed by the immersed sectors by
$$ \CF_{\textrm{imm}}((\bL,\cE),(\bL,\cE)) = \bigoplus_{a} \left(C^*(S_a,\Hom(\cE|_{\iota_-(S_a)},\cE|_{\iota_+(S_a)})) \oplus C^*(S_a,\Hom(\cE|_{\iota_+(S_a)},\cE|_{\iota_-(S_a)}))\right).$$

In this paper, we will use a Morse model \cite{OZ11, BC-pearl, FOOO-can, HKL23} for $\CF((\bL,\cE),(\bL,\cE))$ where we fix a Morse function on each component of $\hat{\bL}$ and the clean intersections $S_a$.\footnote{One can also use de Rham model \cite{FOOO_Kur} where the chains are bundle-valued differential forms, where $m_1 = \nabla + \sum_{\beta \not=0} m_{1,\beta}$ with $\nabla$ being the differential operator induced by the flat connection, and then algebraically pass to a finite dimensional minimal model. We use a Morse model for simpler intuition and computations.} We quickly review the setup here. The chains are linear combinations of homomorphisms of $\cE$ over the critical points of the Morse functions and $m_k$ are defined by counting pearl trajectories weighted by holonomy contributions. To ensure convergence,  we take the base field to be the Novikov field $\Lambda$ (or the base ring to be $\Lambda_0$), where
$$ \Lambda_0 = \left\{ \sum_{i=0}^\infty a_i T^{A_i} \mid a_i \in \C, A_0 \geq 0 \text{ and } A_i \textrm{ is strictly increasing } \right\}, $$
its maximal ideal
$$ \Lambda_+ = \left\{ \sum_{i=0}^\infty a_i T^{A_i} \mid a_i \in \C, A_0 > 0 \text{ and } A_i \textrm{ is strictly increasing } \right\}, $$
and its fraction field
$$ \Lambda = \left\{ \sum_{i=0}^\infty a_i T^{A_i} \mid a_i \in \C, A_i \textrm{ is strictly increasing } \right\}. $$
The $A_\infty$-operations $m_k$ for the category are defined by counting pseudo-holomorphic polygons bounded by $(L_0, L_1, \ldots, L_k)$ weighted by $T^A \,\mathrm{Hol}$, where $A$ denotes the symplectic area of the polygon and $\mathrm{Hol} \in \Hom(\cE_0|_p, \cE_k|_p)$ denotes the holonomy of the flat vector bundles along the counterclockwise boundary of the polygon (also composed with the homomorphisms at input corners) from the output intersection point $p \in L_0 \cap L_k$ back to itself. For a single Lagrangian $\bL$, $m^\bL_0$ is called to be the obstruction of $\bL$, which is a Floer-theoretical generalization of curvature of a bundle over $\bL$.

\begin{rem} \label{rem:trivialization}
	More explicitly, we can write the homomorphisms between components as matrices by choosing simply connected open subsets of components of $\hat{\bL}$ that contain all the critical points in the components and the self clean intersections $S_a$, and a trivialization of the flat bundle $\cE$ over each of these open subset. Then the fibers of $\cE$ over all the critical points can be identified as $\C^d$ for some $d$. The open subset in each component $\hat{\bL}_i$ is taken to be the complement of the union of unstable submanifolds of degree-one critical points of the Morse function on $\hat{\bL}_i$. The construction of Morse model ensures that the unstable submanifolds are transverse to pseudo-holomorphic polygons that contribute to the $m_k$ operations.
\end{rem}

We consider the space of flat connections $\cE$ over $\hat{\bL}$ and bounding cochains in the immersed part of $CF_{\textrm{imm}}^1((\bL,\cE),(\bL,\cE))$ with coefficients in $\Lambda_+$. 
\begin{defn}\label{def: MC}
	Let $\bL$ be a relatively-spin oriented graded Lagrangian immersion and $\bv$ be a dimension vector for flat bundles over $\bL$.
	The set of Maurer-Cartan solutions $\mathrm{MC}(\bv)$ is defined to be
	$$\mathrm{MC}(\bv) := \left\{(A,b): A \in \bigoplus_i \Hom(\pi_1(\hat{\bL}_i),\GL_{\Lambda_{\geq 0}}(v_i)) \textrm { and } b \in CF_{\textrm{imm},\Lambda_+}^1((\bL,\cE_A),(\bL,\cE_A)) \text{ with } m_0^b = 0\right\}$$
	where $\hat{\bL}_i$ are the $i$-th component of $\hat{\bL}$ and $v_i$ is the $i$-th component of $\bv$ which is the rank of $\cE|_{\hat{\bL}_i}$. By the trivialization in Remark \ref{rem:trivialization}, $b$ can be written in terms of tuples of matrices in $\CF_{\textrm{imm},\Lambda_+}^1((\bL,\bv),(\bL,\bv))$, so 
	$$\mathrm{MC}(\bv) \subset \left(\bigoplus_i \Hom(\pi_1(\hat{\bL}_i),\GL_{\Lambda_{\geq 0}}(v_i))\right) \times \CF_{\textrm{imm},\Lambda_+}^1((\bL,\bv),(\bL,\bv))$$
	is an affine variety.
	
	We have the gauge group symmetry
	$\cG$ on $\mathrm{MC}(\bv)$, where $\cG = \prod_i GL_{\Lambda_{\geq 0}}(v_i)$ is a reductive group.
	The Maurer-Cartan space is defined to be the quotient stack 
	$$[\mathrm{MC}(\bv) / \cG]$$ 
	which parametrizes a family of Lagrangian branes $(\cE,b)$ supported on $\bL$.
	
	To be more geometric, we can fix a GIT stability condition $\zeta$ and define the GIT quotient 
	$$\bMC_\zeta(\bv) := \mathrm{MC}(\bv) \sslash_\zeta \cG.$$ 
\end{defn}

By \cite[Theorem 2.19]{CHL17}, we have the $A_\infty$ functor
$$\cF^{(\bL,A,b)}: \Fuk(X) \to dg-Mod(\mathrm{MC}(\bv))$$
defined by the $A_\infty$ operations $m_k^{(\bL,A,b),-,\ldots,-}$, where $dg-Mod$ denotes the dg categories of complexes of modules over the affine variety $\mathrm{MC}(\bv)$.  
Since $m_k^{(\bL,A,b),-,\ldots,-}$ are strictly equivariant under the gauge group action $\cG$ on $\mathrm{MC}(\bv)$, the functor descends to the quotient $\bMC_\zeta(\bv)$.

\begin{thm}
	Let $\bL$ be a relatively-spin oriented graded Lagrangian immersion. For each $\bv \in \Z_{>0}^n$ where $n$ is the number of components of $\hat{\bL}$, there exists an $A_\infty$-functor 
	$$\cF^{(\bL,A,b)}: \Fuk(X) \to \mathrm{Perf}_{dg}(\bMC_\zeta(\bv))$$
	where $\mathrm{Perf}_{dg}$ denotes the dg categories of perfect complexes over $\bMC_\zeta(\bv)$.
\end{thm}

The framed setup below ensures stable quotients.

\begin{defn}\label{def: flag}
	For a relatively-spin oriented graded Lagrangian immersion, suppose its normalization has a component $F$ that has no $H^1$ nor $H^2$. Moreover we assume that for each component $C$, there exists at least one $\deg=1$ morphism between $C$ and a different component. We call the splitting $\bL \cup F$ to be a \emph{framed Lagrangian immersion} $\bL^{\fr}$ where $\bL$ is the union of all components other than $F$. Every bundle $\cE$ we consider over $\bL^\fr$ is required to satisfy $\cE|_{F} = \underline{\C}_F$, the trivial line bundle, so that the gauge group for the component $F$ is trivial. Moreover, by the assumption that $H^2(F)=0$, the framing component $F$ does not contribute to the obstruction $m_0^{\bL^\fr}$.  We denote the Maurer-Cartan space of $\bL^\fr$ by
	$\bMC_\zeta(\bv,\bw).$
\end{defn}

In the framed case, the ambient space $\left(\bigoplus_i \Hom(\pi_1(\hat{\bL}_i),\GL(d_i))\right) \times CF_{\textrm{imm}}^1((\bL^\fr,\bv,\bw),(\bL^\fr,\bv,\bw))$ can be identified with a quiver representation space in the situation of Definition \ref{def: stab}, where the dimension vector $\bw$ is the number of intersections of $F$ with the $i$-th component, and those framing components $a_i$ or $b_i$ that correspond to $\deg\not=1$ morphisms between $F$ and $L_i$ are set to zero. Using the stability condition in Remark \ref{rem:stability}, all semi-stable points are stable points.



\begin{defn} \label{def:Hecke sheaf}
	Let $\bL$ be a relatively-spin oriented graded Lagrangian immersion. For two different rank vectors $\bv^1$ and $\bv^2$, we have the complex of tautological bundles over $\bMC_\zeta(\bv^1) \times \bMC_\zeta(\bv^2)$, which is formed by the Floer complex $\left(\CF^*((\bL,\cE_1), (\bL,\cE_2)),m_1^{b_1,b_2}\right)$ that is equivariant under the gauge group action. The Hecke sheaf $\cH_\zeta(\bv^1,\bv^2)$ is defined to be its cohomology sheaf over $\bMC_\zeta(\bv^1) \times \bMC_\zeta(\bv^2)$.  The sheaf serves as a kernel of the derived Fourier-Mukai transform which gives a functor from the derived category of $\bMC_\zeta(\bv^1)$ to that of $\bMC_\zeta(\bv^2)$.
	
	Similarly, for a framed Lagrangian immersion $\bL^\fr = \bL \cup F$, we have the Hecke sheaf $\cH_\zeta(\bv^1,\bv^2,\bw)$ over $\bMC_\zeta(\bv^1,\bw) \times \bMC_\zeta(\bv^2,\bw)$.
\end{defn}

\begin{defn} \label{def:operators}
	Let $\bL$ be a relatively-spin oriented graded Lagrangian immersion. Fix a character vector $\zeta \in \C^n$ where $n$ is the number of components of $\bL$. Let
	$$ \cD := \bigoplus_\bv D^b(\bMC_\zeta(\bv)). $$
	The creation functor $\cD \to \cD$ is defined by the Fourier-Mukai transform from $D^b(\bMC_\zeta(\bv))$ to $D^b(\bMC_\zeta(\bv+\mathbf{e}^i)$ for all $\bv$, where $\mathbf{e}^i$ is the rank vector that is $1$ on the $i$-th component and $0$ otherwise.
	Similarly, the annihilation functor $\cD \to \cD$ is defined by the Fourier-Mukai transform from $D^b(\bMC_\zeta(\bv))$ to $D^b(\bMC_\zeta(\bv-\mathbf{e}^i))$. ($D^b(\bMC_\zeta(\bv)) := 0$ if $\bv$ has negative entries.)
	
	Similarly, we have $\cD^\fr := \bigoplus_\bv D^b(\bMC_\zeta(\bv,\bw))$ for a framed Lagrangian immersion $\bL^\fr$ and the corresponding creation and annihilation functors $\cD^\fr \to \cD^\fr$.
\end{defn}
Rich and interesting geometric representation theory arise when one considers the algebra generated by these endo-functors and their induced actions on homological or K-theoretic invariants. 
In \cite{HLT24} and in the later part of this paper, $X$ is taken to be a non-compact exact conical symplectic 4-fold. The framing $F$ is topologically $\R^2$ and $\bL$ is a union of spheres.
In particular, when $X$ is a plumbing of $T^*\bS^2$ and $\bL^\fr$ is a framed Lagrangian immersion, we will see that these operations recover the Hecke correspondence algebra.

\subsection{Maurer-Cartan quiver algebra} \label{section:nc mirror}
To work with all dimension vectors $\bv$ simultaneously, we will use the quiver formulation. Noncommutative mirror functors with quivers were formulated in \cite{CHL21}. Motivated by the viewpoint of noncommutative geometry, in which a space is replaced by an algebra and studied via its representations, the commutative Maurer–Cartan space can be interpreted as the moduli of representations of the noncommutative localized mirror, see the discussion at the end of this section.

The quiver formulation modeling flat bundles and boundary deformations in general ranks is given as follows. As in Remark \ref{rem:trivialization}, we have fixed local trivializations of $\hat{\bL}$ and immersed sectors $S_a$ so that the connections and boundary deformations are written as matrices. We use invertible loop arrows at a vertex and arrows between different vertices of a quiver $Q$ to represent them. (This is slightly different from \cite{CHL21}: we include flat connections here and introduce invertible arrows as their holonomy variables. This depends on the choice of Morse functions we have made.) Below are the steps of constructing the Maurer-Cartan quiver algebra $\cA_{\bL}$.

\begin{enumerate}
	\item 	
	The Morse function on each component of $\hat{\bL}$ that we take is required to have a unique maximum point (the Morse flow we take is in the descending convention). Moreover, we take the stable submanifolds of the degree-one critical points as generators of the group $\prod_i \pi_1(\hat{\bL}_i)$. 
	We associate a quiver $Q$ to $\bL$ whose vertices are one-to-one corresponding to components of $\bL$. Each generator of $\pi_1(\hat{\bL}_i)$ corresponds to two loop arrows (back and forth) at the $i$-th vertex; each generator of $\CF_{\mathrm{imm}}^1(\bL)$ (which are critical points of Morse functions on the immersed sectors) corresponds to an arrow between the corresponding vertices.
	\item Let $\Lambda_0 Q$ be the free path algebra of $Q$ over Novikov ring. Each arrow $a$ in $Q$ is associated with a Novikov valuation such that $\text{val}(a) >0$. The valuation induces a filtration on $\Lambda_0 Q$. Take the completion of $\Lambda_0 Q$ with respect to this filtration. By abuse of notation, we still denote the completion by $\Lambda_0 Q$. 
	\item 
	Extend the Fukaya algebra of $\bL$ over the path algebra $\Lambda_0 Q$ and obtain a non-commutative $A_\infty$-algebra $$\tilde{A}^\bL = \Lambda_0 Q\otimes_{\Lambda_0^{\oplus}} \CF(\bL),$$ 
	whose unit is $\one_{\bL} = \sum \one_{\bL_i}$.  $\Lambda_0^{\oplus} \subset \Lambda_0 Q$ denotes $\bigoplus_i \Lambda_0 \cdot e_i$ where $e_i$ are the trivial paths at vertices of $Q$.  The fibered tensor product means that an element $a \otimes X$ is non-zero only when tail of $a$ corresponds to the source of $X$.  The $A_\infty$-operations are defined by
	\begin{equation} \label{eq:mk}
		m_k (f_1 X_1,\ldots,f_k X_k) := f_k \ldots f_1 \, m_k (X_1,\ldots,X_k)
	\end{equation}
	where $X_l \in \CF(\bL)$  and $f_l \in \Lambda_0 Q$.
	\item Extend the formalism of bounding cochains of \cite{FOOO09} over $\Lambda_0 Q$, that is, we take 
	\begin{equation} \label{eq:b}
		\mathbf{A} = (a_\gamma), \,\bb = \sum_l b_l X_l
	\end{equation}
	where $\gamma$ are the generators of $\pi_1(\hat{\bL}_i)$, $a_\gamma$ are the corresponding self-arrows in $Q$, $X_l$ are the immersed generators of $\CF^1(\bL)$, and $b_l$ are the corresponding arrows in $Q$.  Then define the deformed $A_\infty$-structure $m_k^{\mathbf{A},\bb}$ as in \cite{FOOO09} and via Equation \eqref{eq:mk}. To simplify the notation, we denote $(\mathbf{A},\bb)$ just by $\bb$. 
	\item Quotient out the quiver algebra by the two-sided ideal $R$ generated by coefficients of the obstruction term $m_0^{\bb}$, together with $xy-e_i$ for each pair of back-and-forth loop arrows $x,y$ associated to a generator of $\pi_1(\hat{\bL}_i)$, and those that correspond to the relations for $\pi_1(\hat{\bL}_i)$.
	$$\cA_\bL := \Lambda_0 Q/R.$$ 
	\item Extend the Fukaya category over $\cA_\bL$, and enlarge the Fukaya category by including the noncommutative family of objects $(\bL,\bb)$ where $\bb$ in \eqref{eq:b} is now defined over $\cA_\bL$.  This means for $L_1,L_2$ in the original Fukaya category, the morphism space is now extended as $\cA_\bL \otimes \CF(L_1,L_2)$.  The morphism spaces between $(\bL,\bb)$ and $L$ are enlarged to be $\CF((\bL,\bb),L) := \cA_\bL\otimes_{\Lambda_0^{\oplus}} \CF(\mathbb{L},L)$ (and similarly for $\CF(L,(\bL,\bb))$).  We already have $\CF((\bL,\bb),(\bL,\bb))$ in Step 2 (except that $\Lambda_0 Q$ is replaced by $\cA_\bL$).  The $m_k$ operations are extended in a similar way as \eqref{eq:mk}.
\end{enumerate}

\begin{thm}[\cite{CHL21}]
	Consider the boundary deformation $\bb$ of $\bL$. There exists a well-defined $A_\infty$-functor
	$$\cF^{(\bL,\bb)}: \Fuk(X) \to \MF(\cA_\bL,W).$$
\end{thm}

\begin{rem}
	$m_0^{\bb}=W\cdot \one_\bL$ has degree $2$.  Thus if furthermore $\bL$ is graded and we work in the $\Z$-graded case, $W=0$, and the above $\MF(\cA_\bL,W)$ reduces to the dg category of complexes of $\cA_\bL$-modules.
	
	Intuitively, the noncommutative deformation space $\cA_\bL$ can be understood via Strominger-Yau-Zaslow Conjecture \cite{SYZ96}, which predicts that the mirror space is constructed as the moduli space of (special) Lagrangians (see also \cite{Aur07}). Roughly, a Lagrangian $\bL$ corresponds to a point of the mirror, while its deformation space $\cA_\bL$ forms a neighborhood of that point. Thus, $\cA_\bL$ is also refered as the localized mirror or noncommutative deformation space.
\end{rem}

The connection between the noncommutative deformation space $\cA_\bL$ and the Maurer-Cartan space $\bMC_\zeta(\bv)$ is as follows.
\begin{prop}
	The Maurer-Cartan space $\bMC_\zeta(\bv)$ equals the moduli space of $\zeta$-semistable representations of the quiver algebra $\cA_\bL$.
\end{prop}

\begin{proof}
By Remark \ref{rem:trivialization}, we have taken local trivializations to write $(A,b) \in \mathrm{MC}(\bv)$ in terms of matrices $A_D$ and $B_X$ respectively, where $D$ are the degree-one critical points of components of $\hat{\bL}$ and $X$ are the degree-one immersed generators. Thus, $(A^{\pm 1},b)$ can be taken as a representation of the arrows of $Q$. The obstruction equation $m_0^{(A,b)}=0$ for $(A,b)$ is identical to the quiver version $m_0^{\bb}=0$ via the quiver representation of $\bb$ by $(A,b)$. Moreover, the gauge group $\cG$ for ${MC}(\bv)$ equals to the automorphism group of quiver representations. Consequently, the Maurer-Cartan space $\bMC_\zeta(\bv)$ can be identified with the moduli space of $\zeta$-semistable representations of $\cA_\bL$. 
\end{proof}

\subsection{Lagrangian Floer cohomology and quiver representations}

In this subsection we establish some general relations between quiver representations and Lagrangian Floer theory. 

In previous subsections, we have introduced the moduli space $\bMC_\zeta(\bv)$ of Lagrangian branes with flat connections in rank $\bv$ supported over $\bL$.
For the purpose of this subsection, we will work with a slightly different moduli space $\widehat{\bMC}_\zeta(\bv)$ that takes degree-one boundary deformations in place of flat connections. 
When all components of $\bL$ are simply connected (which is the case in the later sections), there is no non-trivial flat connection and so $\bMC_\zeta(\bv) = \widehat{\bMC}_\zeta(\bv)$.

Roughly speaking, the relation between $\widehat{\bMC}_\zeta(\bv)$ and $\bMC_\zeta(\bv)$ is analogous to that between the Lie algebra $\mathfrak{gl}(n)$ and the Lie group $\GL(n)$.
However, the precise relation is subtle since $\pi_1(\bL)$ has non-commuting relations which are not directly related to those produced from $m_0^b$ by boundary deformations $b$.
For the proposition below, we need to use the relation between $m_0^b$ and $m_1^b$ which only works for $\widehat{\bMC}_\zeta(\bv)$ but not its multiplicative version $\bMC_\zeta(\bv)$.

\begin{defn}\label{def: MChat}
	Let $\bL$ be a relatively-spin oriented graded Lagrangian immersion equipped with the trivial bundle in rank $\bv$.
	The set of Maurer-Cartan solutions $\widehat{\mathrm{MC}}(\bv)$ is defined to be the affine variety
	$$\widehat{\mathrm{MC}}(\bv) := \left\{b \in CF_{\Lambda_+}^1((\bL,\bv),(\bL,\bv)): m_0^b = 0\right\}.$$
	
	We have the gauge group symmetry
	$\cG$ on $\widehat{\mathrm{MC}}(\bv)$, where $\cG = \prod_i GL_{\Lambda_{\geq 0}}(d_i)$ is a reductive group.
	We fix a GIT stability condition $\zeta$ and define the GIT quotient 
	$$\widehat{\bMC}_\zeta(\bv) := \widehat{\mathrm{MC}}(\bv) \sslash_\zeta \cG.$$ 
\end{defn}

$\widehat{\bMC}_\zeta(\bv)$ is a quiver variety with generic stability condition $\zeta$. 
$[b] \in \widehat{\bMC}_\zeta(\bv)$ is an equivalent classe of quiver representations, or equivalently a finite-dimensional module of the quiver algebra $\cA = \cA_{\bL}$. We write $\mathrm{Ext}^*_{\cA}(b_1,b_2)$ for the Ext-groups between such modules.  Then
\begin{prop}\label{prop: Hom-Ext}
     Let $b_k \in \widehat{\bMC}_\zeta(\bv_k)$ for $k=1,2$, and suppose $\cA = \cA_\bL$ is admissible (in the sense that each path in the defining quiver relations has at least length two). Then there are isomorphisms
	\begin{align}
	\HF^0((\bL,b_1), (\bL,b_2))\cong&\, \Hom_{\cA}(b_1,b_2); \label{eq:HF0} \\ 
	\HF^1((\bL,b_1), (\bL,b_2))\cong&\, \mathrm{Ext}^1_{\cA}(b_1,b_2). \label{eq:HF1}
	\end{align}
\end{prop}
\begin{proof}
	Consider $(\CF^*((\bL,\cE_1),(\bL,\cE_2)),m_1^{b_1,b_2})$, where $\cE_k$ is the trivial bundle of rank $v_k$ over $\bL$. Let $M_i \in \CF^0(\bL,\bL)$ be the unique maximal point of the Morse function on the $i$-th irreducible component which represents the unit of $\bL_i$. More precisely, we follow the homotopy unit construction for the Morse model \cite{FOOO09,CW15,KLZ}, which produces the unit $\one_i$ as a degree $0$ cochain and a homotopy element $h_i$ as a degree $(-1)$ cochain with the differential $d h_i = \one_i - M_i$. Then we take the truncated cochain complex (that has the same cohomology) whose degree-zero part is taken to be the cokernel of the image of $h_i$. We write $(\CF^*((\bL,\cE_1),(\bL,\cE_2)),m_1^{b_1,b_2}):= (\oplus_{p} \Hom({\cE_1|_{t(p)},\cE_2|_{h(p)}})X_{p},m_1^{b_1,b_2})$. 
	
	Denote the deformation parameter of $(\bL,\cE_k)$ by $b_k= \sum_a B_a^kX_a$ for $k=1,2$. 
	Let $V^1= \bigoplus_i \cE_1|_{M_i}$ and $V^2= \bigoplus_i \cE_2|_{M_i}$. Then $(V^1,b_1)$ and $(V^2,b_2)$ give representations of $\cA$. 
	In the following, we want to show  $\HF^0((\bL,\cE_1,b_1),(\bL,\cE_2,b_2))= \Hom_{\cA}(V^1,V^2)$ and $\HF^1((\bL,\cE_1,b_1),(\bL,\cE_2,b_2))=\mathrm{Ext}^1_{\cA}(V^1,V^2).$
	
	Using the unital property of $\sum_i M_i$, we have
	\begin{align*}
		m_1^{b_1,b_2}(\sum_i \xi_i M_i)&= m_2(b_1, \sum_i \xi_i M_i) + m_2(\sum_i \xi_i M_i,b_2)\\
		&= \sum_i \left(\sum_{t(a)=i} (B_a^2 \xi_i)X_a - \sum_{h(c)=i} (\xi_i B_c^1)X_c \right)\\
		&=\sum_a \left(B_a^2 \xi_{t(a)}- \xi_{h(a)}B_a^1  \right)X_a
	\end{align*}
	 where $\xi_i \in \Hom(\cE_1|_{M_i},\cE_2|_{M_i})$ and $X_a$ are generators of $\CF^1(\bL,\bL)$. Therefore, $\HF^0((\bL,\cE_1),(\bL,\cE_2))$ consists of linear maps $(\xi_i)$ such that $B_a^2 \xi_{t(a)}= \xi_{h(a)}B_a^1$, which is the morphism space $\Hom_{\cA}(b_1,b_2)$. This proves \eqref{eq:HF0}.
	 
	 The complex $(\CF^*((\bL,\cE_1),(\bL,\cE_2)),m_1^{b_1,b_2})$ arises from tensoring $\CF^*((\bL,\bb),(\bL,\bb))$ with the corresponding $\cA$-modules, namely, $\CF^*((\bL,\cE_1),(\bL,\cE_2)) \cong (V^1)^* \otimes_{\cA_{\bL}} \CF^*((\bL,\bb),(\bL,\bb)) \otimes_{\cA} V^2$.
	 Observe that 
	 $$m_1^{b_1,b_2}(\xi X_a)=\sum m_{k}(b_1,\ldots, b_1, \xi X_a, b_2,\ldots,b_2)= \iota^{b_1,b_2}_\xi \partial_{x_a} m_0^{\bb},$$ where $X_a \in \CF^1((\bL,\bb),(\bL,\bb))$, $x_a$ is the corresponding arrow, 
	 $\partial_{x_a} m_0^{\bb}$ and $\iota^{b_1,b_2}_\xi$ are defined as follows. Let $P_{ij}$ be the bimodule $\cA e_i \otimes e_j \cA$ for all vertices $i,j$. 
	 Then we define the operator $\partial_{x_a}: kQ \to P_{h(a)t(a)}$ that symbolically replace $x_a$ by the tensor symbol $\otimes$ by the Leibniz rule. 
	 Then $\iota^{b_1,b_2}_\xi: \cA e_i \otimes e_j \cA \to \Hom(V^1_i,V^2_j)$ is defined by the substitution $\bb = b_1$ for the right $\cA$, $\bb = b_2$ for the left $\cA$, and replacing the tensor symbol by $\xi$.
   	
    To compute $\mathrm{Ext}^1_{\cA}(b_1,b_2)$, we can use the following bimodule resolution of the admissible path algebra $\cA$, see for example \cite{Ber07}:
   	$$P_\bullet :=[\cdots \to \bigoplus_{R} P_{h(R)t(R)}\xrightarrow{g} \bigoplus_a P_{h(a)t(a)} \xrightarrow{f} \bigoplus_i P_{ii}] \twoheadrightarrow \cA \to 0,$$ where $R$ is the relation of $\cA$. Here $f_a: P_{h(a)t(a)} \to \bigoplus_i P_{ii}$ takes $e_{h(a)} \otimes e_{t(a)} \mapsto a \otimes e_{t(a)}- e_{h(a)}\otimes a$ and $f=\sum_a f_a$, while $g_R: P_{h(R)t(R)} \to \bigoplus_a P_{h(a)t(a)}$ takes $e_{h(R)} \otimes e_{t(R)} \mapsto \sum_a\partial_a R$ and $g=\sum_R g_R$. 
   	
   	Given any left $\cA$-module $M$, the complex $P_\bullet \otimes_{\cA} M$ provides a projective resolution of $M$. By direct comparison, the first two terms of the induced differentials for $\R\Hom_{\cA}(V^1,V^2)$ coincides with that of $(\CF^*((\bL,\cE_1),(\bL,\cE_2)),m_1^{b_1,b_2})$. Hence, the result holds.
\end{proof}

\begin{rem}
	Geometrically, the admissibility condition of $\cA$ requires that one-gons and two-gons have zero sum contribution to $m_0^{\bb}$. The condition is used only in proving \eqref{eq:HF1} but not \eqref{eq:HF0}.
\end{rem}

\begin{rem}
	If we take $\bMC_\zeta(\bv)$ in place of $\widehat{\bMC}_\zeta(\bv)$, we can still prove an analogous isomorphism as \eqref{eq:HF0} by the unital property:
	$$ \HF^0((\bL,A_1,b_1), (\bL,A_2,b_2))\cong \Hom_{\cA}((A_1,b_1),(A_2,b_2)). $$
\end{rem}

 
In the rest of this subsection, we only use \eqref{eq:HF0}, and so we can either take $\bMC_\zeta(\bv)$ or $\widehat{\bMC}_\zeta(\bv)$. For convenience we only make the statements for $\widehat{\bMC}_\zeta(\bv)$. The following is immediate from the $A_\infty$-operations over the Maurer-Cartan spaces.
\begin{cor}
	Let $b_k \in \widehat{\bMC}_\zeta(\bv_k)$ for $k=1,2,3$. 
	$$m_2^{b_1,b_2,b_3}:\HF^0((\bL,b_1), (\bL,b_2)) \otimes \HF^0((\bL,b_2), (\bL,b_3)) \to \HF^0((\bL,b_1), (\bL,b_3))$$ coincides with the composition of morphisms in $\Hom_{\cA}(b_i,b_{i+1})$ for the quiver algebra $\cA$.
\end{cor}



Recall that in the construction of cohomological Hall algebras \cite{KS11}, an essential ingredient is the moduli stack (or space) of extensions. Lagrangian Floer theory admits an analogous structure once the following assumption is satisfied. We will see that this is the case for framed Lagrangians constructed by plumbing of two-spheres.

\begin{assump}
	There exists stability conditions $\zeta_{\bv}$ for $\widehat{\bMC}(\bv)$ for all rank vectors $\bv$ such that for any
	$b_i \in \widehat{\bMC}_{\zeta_{\bv_i}}(\bv_i)$ $(i=1,2)$, every nonzero homomorphism in $\HF^0((\bL,\bv_1,b_1),(\bL,\bv_2,b_2)) \cong \Hom_{\cA}(b_1,b_2)$ is \emph{injective} if $\bv_1 \le \bv_2$, and surjective if $\bv_1 \ge \bv_2$.
\end{assump}

\begin{prop}
   	Let $(\bL,b_k)$ be a family of Lagrangian branes of rank $\bv^k$ for $k=1,2,3$ where $\bv^2=\bv^1 +\bv^3$. Under the above assumption, the fiberwise kernel of 
   	$$m_2: \HF^0((\bL,b_1), (\bL,b_2)) \otimes \HF^0((\bL,b_2), (\bL,b_3)) \to \HF^0((\bL,b_1), (\bL,b_3)), $$ 
   	as a sheaf over $\prod_{k=1}^{3}\widehat{\bMC}_\zeta(\bv^k)$, is supported on the locus consisting of $(b_1,b_2,b_3)$ such that 
   	$$0\to b_1 \hookrightarrow b_2 \xrightarrowdbl{} b_3 \to 0$$ 
   	forms a short exact sequence of $\cA_L$-modules. 
   	In other words, for any fixed $b_1$ and $b_3$, the fiberwise kernel is supported precisely on the extensions of $b_3$ by $b_1$.
\end{prop}
\begin{proof}
	Under the above assumption, any nonzero elements $f$ and $g$, where $$(f,g) \in \HF^0((\bL,\cE_1,b_1),(\bL,\cE_2,b_2)) \otimes \HF^0((\bL,\cE_2,b_2), (\bL,\cE_3,b_3)),$$ determine a sequence of $\cA_L$-modules $b_1 \xhookrightarrow{f} b_2 \xrightarrowdbl{g}b_3$. It remains to show $g\circ f=0$ and $\mathrm{Ker}(g)=\mathrm{Im}(f)$, which follow from the condition that $m_2^{b_1,b_2}(f,g)=0$ together with the dimension constraint.
\end{proof}

\subsection{Holomorphic symplectic structure on Maurer-Cartan space of a Lagrangian immersed surface}

Now, let's focus on $\dim \bL = 2$. For a compact Riemann surface of genus $g$, the moduli space of $G=\GL(n,\C)$-flat connections called the character variety has a canonical holomorphic symplectic structure \cite{Atiyah-Bott, Goldman}. Its fundamental group can be written as $\langle a_1,\ldots,a_g, a_1',\ldots,a_g' \rangle / (\prod_{i=1}^g a_ia_i'a_i^{-1}(a_i')^{-1})$, so that the moduli space of flat connections consists of tuples of matrices $A=(A_i,A_i':i=1,\ldots,g) \in G^{2g}$ satisfying the relation $\prod_{i=1}^g A_iA_i'A_i^{-1}(A_i')^{-1} = \Id$, modulo conjugation $gAg^{-1} = (gA_ig^{-1},gA_i'g^{-1}:i=1,\ldots,g)$. To fit with our previous setup, we can fix a perfect Morse function whose unstable submanifolds of $\deg=1$ critical points are these cycles $a_i,a_i'$.

The symplectic structure on the moduli space can be understood by the beautiful theory of quasi-Hamiltonian manifolds and group-valued moment maps found by \cite{AMM} for compact groups and \cite{Bursztyn-Crainic, ABM} for non-compact Lie groups using Dirac structures. We give a very short review below.

For a Lie group $G$, the double $\mathbf{D}(G)=G\times G$ is equipped with the two-form
$$ \omega_{\mathbf{D}(G)} = \frac{1}{2}\left(a^*\theta^L \cdot b^*\theta^R + a^*\theta^R\cdot b^*\theta^L + (ab)^*\theta^L \cdot (a^{-1}b^{-1})^*\theta^R\right) $$
where $\theta^L,\theta^R$ are the left-invariant and right-invariant Maurer-Cartan $\mathfrak{g}$-valued one-form of $G$ respectively, the dot is induced by the Killing form of $\mathfrak{g}$, and $a,b:\mathbf{D}(G) \to G$ are the projections to the left and right factors respectively. Then $\Phi: \mathbf{D}(G) \to G$ defined by $\Phi(a,b) = aba^{-1}b^{-1}$ satisfies the group-valued moment map conditions with respect to $\omega_{\mathbf{D}(G)}$.

The quasi-Hamiltonian structure on the product $(\mathbf{D}(G))^g$ is constructed by the fusion product. For two quasi-Hamiltonian manifolds $(M_i,\omega_i,\Phi_i)$, their fusion is defined as
$$ (M_1\times M_2, \omega_1+\omega_2+\frac{1}{2}\Phi_1^*\theta^L\cdot \Phi_2^*\theta^R,\Phi_1\Phi_2). $$
For the quasi-Hamiltonian manifold $(\mathbf{D}(G)^g,\omega,\Phi)$, the GIT quotient $\Phi^{-1}\{\Id\} \sslash_\zeta G$ gives the moduli space of $G=\GL(n,\C)$-flat connections as a holomorphic symplectic variety.

Next, we consider a nodal immersed Lagrangian as in Definition \ref{def: MChat}, whose normalization is a union of Riemann surfaces. We take a perfect Morse function in each component as stated above. For the consideration of this subsection, we work over $\C = \Lambda_{\geq 0} / T^{>0}$, in other words, we only consider contributions of constant polygons in $m_k^b$ whose coefficients are taken as formal power series in the matrix coefficients of $b$. We denote the resulting $A_\infty$ operations by $m_{k,0}^b$. 

Using the result of \cite{FOOO-involution} and as in \cite{HKL23,HLT24}, for the moduli spaces of constant polygons at each nodal point, we have taken the Kuranishi perturbations that respect the local anti-symplectic involution around the nodal point, so that the discs in $m_{2k,0}(U,V,\ldots,U,V)$ are related with those in $m_{2k,0}(V,U,\ldots,V,U)$ by a negative sign.

\begin{prop}
	$\omega(X,Y) := \tr (m^b_{2,0}(X,Y),\one)$ gives a formal symplectic form (valued in formal power series) on $CF^1((\bL,\bv),(\bL,\bv))$. Moreover, $\mu = (m_{0,0}^b,\one)$ gives a moment map and hence $\widehat{\bMC}_\zeta(\bv)$ has a (formal) symplectic structure.
\end{prop}
\begin{proof}
	By the local anti-symplectic involution symmetry and the fact that $m_{k,0}(X_1,\ldots,X_k)=0$ whenever two adjacent inputs are equal ($X_i=X_{i+1}$ for some $i$), we have $\omega(X,Y)=\omega(Y,X)$. Moreover, the contribution in $(m_0^b,\one)$ of the immersed sectors $X$ and $Y$ at a nodal point is of the form
	$$ \left(\left(1+ \sum_{i=1}^{\infty}  a_i(XY)^i\right)XY-\left(1+ \sum_{1}^{\infty}  a_i(YX)^i\right) YX\right) $$
	for some $a_i \in \Q$. 
	By a change of coordinates, we can simplify $(m_{0,0}^b,\one)$ to $\sum_{a} (\tilde{X}_a\tilde{Y}_a - \tilde{Y}_a\tilde{X}_a)$ (where $\tilde{X}_a,\tilde{Y}_a$ are matrices indexed by the nodal points $a$). Since $m_{2,0}^b(X,Y)$ arises as the formal second derivative of $m_{0,0}^b$, $\tr (m_{2,0}^b(X,Y),\one)$ becomes the standard symplectic form $\sum_a \tr(d\tilde{X}_a \wedge d\tilde{Y}_a)$. Thus, $\omega$ is closed and non-degenerate. The change of coordinates is expressed in formal series of matrices, and hence is gauge-equivariant. In such a standard coordinate, it is obvious that $\mu = (m_{0,0}^b,\one)$ is the moment map. Then $\widehat{\bMC}_\zeta(\bv)$ has the quotient symplectic structure.
\end{proof}

\begin{rem}
	While the Hamiltonian structure is coordinate-free, if the Floer theory is made to respect cyclic symmetry \cite{Fuk10}, then $(m_{0,0}^b,\one)$ and $\tr (m_{2,0}^b(X,Y),\one)$ will have the standard form.
\end{rem}

Now we consider the moduli space $\bMC_\zeta(\bv)$ that parametrizes both flat connections and Maurer-Cartan elements at the immersed sectors.  Using quasi-Hamiltonian reduction, we can prove the following.
\begin{thm}
	There exists an open subset $\bMC^\circ_\zeta(\bv) \subset \bMC_\zeta(\bv)$ containing all the flat connections and the zero Maurer-Cartan elements such that it has a holomorphic symplectic structure.
\end{thm}
\begin{proof}
	This is the quotient by $\prod_i GL(d_i,\C)$ of the subvariety in $(\prod_i G^{2g_i})\,\times\, \CF_{\textrm{imm}}^1((\bL,\bv),(\bL,\bv))$ defined by the equations $\prod_{p=1}^{g_i} A_pA_p'A_p^{-1}(A_p')^{-1} = \Id$ for components $\bL_i$ and $\mu(b)=(m_{0,0}^b,\one) = 0$. 
	
	To apply the quasi-Hamiltonian reduction procedure, we compose $\mu$ with the exponential map to get a group-valued moment map $\exp \circ \mu$, where $\CF_{\textrm{imm}}^1((\bL,\bv),(\bL,\bv))$ is equipped with the two-form $\omega + \mu^*\tau$ where $\omega$ is defined by $m_{2,0}^b$ as above and $\tau$ is the primitive two-form on $\mathfrak{g}\cong \mathfrak{g}^*$ that satisfies $d\tau = \exp^* \eta$ for the Cartan three-form on $G$. $\exp \circ \mu$ is a group-valued moment map on a neighborhood $U \ni 0$ in $\CF_{\textrm{imm}}^1((\bL,\bv),(\bL,\bv))$, whose image under $\mu$ is in the open subset $\mathfrak{g}_{\sharp} \subset \mathfrak{g}$ on which $\exp$ is a local diffeomorphism. 
	
	By the fusion product of quasi-Hamiltonian manifolds, we obtain a quasi-Hamiltonian structure on $\prod_i G^{2g_i} \times U$ with the group-valued moment map given by the product of the individual moment maps. The quasi-Hamiltonian reduction at the identity then gives a symplectic structure on the quotient $\bMC^\circ_\zeta(\bv)$.
\end{proof}

\section{Construction of Hecke correspondence from Lagrangian Floer theory} \label{sec: H}

Having set up the general framework, we now focus on Nakajima quiver varieties produced from Lagrangian branes in a convex non-compact symplectic 4-fold. We will construct Hecke correspondences and holomorphic Lagrangians via mirror symmetry.

\subsection{Framed Lagrangian branes} \label{sec: framed}
In this subsection, we briefly recall the definition of framed Lagrangian branes in a convex non-compact symplectic manifold introduced in \cite{HLT24}. The main statement is that all Nakajima quiver varieties are the Maurer-Cartan deformation spaces of certain framed Lagrangian branes. 

Let $(M,\omega)$ be a (non-compact) convex symplectic manifold of dimension $2n$ (see for instance \cite[(7b)]{Sei08} in the exact setting, and \cite[Section 3]{RS2017} in the monotone setup). In other words, $M$ is assumed to be conical at infinity: there exists a compact subset $M_c \subset M$ such that its complement can be identified with an infinite cone by a symplectomorphism 
\begin{align*}
	\phi: \left(M\setminus M_{c}^{\circ},\omega\right)  \cong  \left(\partial M_{c}\times [1,\infty), d(r\alpha)\right)
\end{align*}
where $M_{c}^{\circ}$ is the interior of $M_{c}$, $(\partial M_{c},\alpha)$ is a contact manifold and $r$ is the coordinate of $[1,\infty)$.  Moreover,
\begin{enumerate}
	\item The Liouville vector field $Z$ characterized by $\omega(Z,\cdot) = \theta$ points strictly outwards along $\partial M_{c}$. 
	
	\item $\phi$ is induced by the flow of $Z$ for time $\log r$.
	
	\item Outside the compact set $M_{c}$, the symplectic form is exact: $\omega = d\theta$ where $\theta = \phi^*(r\alpha)$.  
	
\end{enumerate}

A conical almost complex structure $J$ is taken, which is an $\omega$-compatible almost complex structure on $M$ satisfying $J^*\theta = dr$ for $r$ large enough. 

Recall that we have defined framed Lagrangian immersion in Definition \ref{def: flag}, which is a union of a compact Lagrangian immersion $\bL$ together with a framing $F$. In the convex setup, we make a further assumption on the framing $F$ as follows.

\begin{defn}[Conical framing Lagrangian] \label{def:framing} A conical framing Lagrangian of a convex symplectic manifold $M$ is a Lagrangian $F\subset M$ diffeomorphic to $\mathbb{R}^n$  equipped with a Morse function $f_F$ with exactly one maximal point and no other critical point, and $F$ is of the form $F = F_{c} \cup_{\partial F_{c}}(\partial F_{c}\times [1,\infty))$ where 
	\begin{enumerate}
		\item $F_{c}:=F\cap M_c$ intersects $\partial M_{c}$ transversally, 
		\item $\theta|_{F-F_{c}}$ vanishes near the boundary $\partial F_{c} := \partial M_{c}\cap F_{c}.$  In particular, the complement $F-F_c$ is an exact Lagrangian with a constant primitive function.
	\end{enumerate} 
\end{defn}

\begin{defn}
	A framed convex symplectic manifold is a tuple $(M,\bF)$ where $M$ is a convex symplectic manifold and $\bF$ is a collection of framing Lagrangians that are disjoint with each other.  We denote by $|\bF|$ the cardinal number of the framing collection. 
\end{defn}

\begin{example} \label{exmp:plumbing}
	Let's recall the  plumbing  construction of an exact symplectic manifold \cite[Section 2.3]{MR2786590}, which is automatically a Liouville domain whose completion is a convex symplectic manifold. 
	It is obtained from gluing by a symplectomorphism between two disk bundles $D^*S_{1}$ and $D^*S_{2}$ over two spheres that identifies the fiber and base in opposite directions.
	
	Let $M_1,M_2$ be compact Riemannian manifolds, and we consider the total spaces of their disk cotangent bundles $D^*M_1, D^*M_2$.  Consider open balls $B_i \subset D_i$ such that the disk bundles $D^*B_i \subset D^*M_i$ for $i=1,2$ are symplectomorphic under  $(x_{1},y_{1})  \mapsto (-y_{2},x_{2})$, where $x_i,y_i$ denote the base and fiber coordinates respectively.
	The glued manifold 
	can be completed to a Liouville manifold (after rounding off the corner).  Moreover, $M_{i}$ are are embedded exact Lagrangian submanifolds. 
	In the resulting Liouville manifold, take a cotangent fiber of $M_i$ as a framing Lagrangian; then the union of $M_i$ with the cotangent fibers yields a framed Lagrangian immersion.
\end{example}

We will be particularly interested in the case that $M$ is the plumbing space of $T^*\bS^2$ according to a graph $D$.

\begin{thm}\cite[Thm. 5.2]{HLT24}
	Let $M$ be the plumbing space of $T^*\bS^2$ according to a graph $D:=(I,E)$, and 
	$L^{\fr}\;=\;L\cup_{v\in I} F_v$
	be the framed Lagrangian given by the zero section $L$ together with a cotangent fiber $F_v$ in the component indexed by vertex $v$. Then quiver associated to $L$ (resp. $\bL^\fr$) is the (resp. framed) double quiver of $D$. Moreover, the noncommutative deformation space $\cA_L$ (resp. $\cA_{\bL^\fr}$) of $L$ (resp. $\bL^\fr$) is isomorphic to the (resp. framed) preprojective algebra of $D$.
\end{thm}

A key ingredient is the local involution symmetry introduced in \cite[Lem. 3.4]{HKL23}, which enables a change of coordinate identifying the obstruction equations $m_0^{\bb}=0$ with the preprojective relations. Thus, $\cA_L$ is isomorphic to the preprojective algebra.

\begin{cor}\label{cor: MC}
	Take $M$ and $\bL^\fr$ as above. Let $\cE$ be a trivial vector bundle with the trivial connection of rank $\bv$ (resp. rank $\bw$) over the compact components (resp. framing components) of $\bL^\fr$. Then restricting to the Fukaya subcategory generated by $\bL^\fr$, the Maurer-Cartan space of the framed Lagrangian brane $\bMC_\zeta(L^{\fr},\cE)$ is isomorphic to the Nakajima quiver variety $\bM_\zeta(\bv,\bw).$
\end{cor}

Due to the existence of the coordinate changes, the objects of $\cM_\zeta(\bv,\bw)$ only have polynomial dependence in $T$. Hence, one can specialize to obtain the complex–valued theory and thereby reduce the constructions the setting over $\C$. 

In the remaining sections, we will apply the same strategy: we restrict the localized mirror functor to the subcategory that has no convergence issues and work over the complex number.

Finally, we note that in \cite{HLT24}, the authors computed the image of a framed Lagrangian brane $(\bL^\fr,\cE,b_0)$ under the localized mirror functor $\cF^{(\bL,\bb)}(\bL^\fr,\cE,b_0)$, which is isomorphic to the monadic complex of the torsion free sheaf over $\widetilde{\C^2/\Gamma}$ in the affine $ADE$ cases. 

\begin{thm}\cite[Thm. 5.13]{HLT24}
	$\cF^{(\bL,\bb)}$ transforms the framed Lagrangian branes $(L^{\textrm{fr}},\cE,b_0)$ into the following complexes:
	\begin{equation} \label{eq:univb}
		(\cF^{(\bL,\bb)}(L^{\textrm{fr}},\cE,b_0),m_1^{\bb,b_0}):  L^0((\bL,\bb),(L,\cE)) \to E((\bL,\bb),(L,\cE)) \oplus L^1((\bL,\bb), (F,\cE)) \to L^2((\bL,\bb),(L,\cE)),
	\end{equation}
    which coincides with the monad constructed algebraically by Nakajima \cite{Nak07}. 	This complex is written using the notations introduced in Remark \ref{rem: tau}, and its terms are spanned by Floer generators $M_v \in L^0((\bL,\bb),(L,\cE)), P_v \in L^2((\bL,\bb),(L,\cE))$, and the immersed sectors $J_v \in L^1((\bL,\bb), (F,\cE))$, $X_a \in E((\bL,\bb),(L,\cE))$ for all vertices $v$ and arrows $a$ of $Q$. Here, $M_v$, $P_v$ are the maximal points and minimum points of the Morse function $f_v$ on $\bS^2_v$ respectively, while $X_a \in \CF^1(\bS^2_{t(a)},\bS^2_{h(a)})$ and $J_v \in \CF^1(\bL_v, F_v)$ are degree $1$ immersed sectors. 
	Furthermore, the first differential takes the form $$m_{1}^{\bb,b_0}(\eta M_v)= \sum_{t(\bar{a})=v} (B_{\bar{a}} \eta) X_{\bar{a}} + \sum_{h(a)=v} ( \eta x_a )X_a + j_v\eta J_v,$$ and the second one is $$m_{1}^{\bb,b_0}(\eta' X_a)=  \eta' x_{\bar{a}} P_{h(a)}+ B_{\bar{a}} \eta' P_{t(a)}; \,\,  m_{1}^{\bb}(\eta'' J_v) = i_v \eta'' P_v.$$
	
	Furthermore, the cohomology of this complex is isomorphic to a torsion free sheaf over $\widetilde{\C^2/\Gamma},$ which admits an extension to a framed torsion-free sheaf over $\widetilde{\bP^2/\Gamma}.$
\end{thm}

In the next subsection, we will use a more general monadic complex $\left(\CF^*((\bL^\fr,\cE_1),(\bL^\fr,\cE_2)), m_1^{b_1,b_2}\right)$ to construct the Hecke correspondence.

\subsection{Hecke Correspondence}\label{sec:Hecke}

Hecke correspondence plays an important role in the geometric representation theory, which serves as the Kernel of the Fourier-Mukai type transforms in defining the creation and annihilation operators \cite{Nak98,Nak01}. 
In this subsection, we realize their Lagrangian Floer-theoretic analogue. 

In the following, we always assume $D$ is a graph without self-loops i.e. there are no arrows whose head and tail are the same vertex. This technical assumption ensures that the corresponding subvariety has the correct dimension. We denote by $Q$ the double quiver of $D$, and by $Q^\fr$ its framed double quiver.

Let $\bL^\fr$ be the framed Lagrangian immersion. Take the dimension vectors $\mathbf{w}$,  $\mathbf{v}^1$ and $\mathbf{v}^2$. Let $(\bL^\fr,\cE_m)$ be the stable framed Lagrangian brane of rank $(\mathbf{v}^m,\mathbf{w})$ with the deformation parameter $b_m:=\sum B_{ma}X_a + \sum i_{mv} I_v + \sum j_{mv} J_v $ for $m=1,2$. Here $B_m, i_m, j_m$ should be viewed as the data of the Nakajima quiver variety, appearing naturally as coefficients in the deformation parameter.

Recall that the Maurer-Cartan space of a framed Lagrangian brane $(\bL^\fr,\cE)$ is isomorphic to a Nakajima quiver variety $\bM_\zeta(\mathbf{v},\mathbf{w})$ under a suitable stability condition. In what follows, we will always adopt the stability condition described in Remark \ref{rem:stability}, and denote the Maurer-Cartan space by $\bMC(\bv,\bw)$ or $\bM(\bv,\bw)$. Besides, we will often write the deformation parameters as $b_m:=[B_m,i_m,j_m]$ for short, since $b_m \in \bMC(\bv^m,\bw) $ is an equivalent class of quiver representation, see Section \ref{sec:op_Fuk}.

We consider the complex of vector bundles $(\CF^*((\bL^\fr,\cE_1),(\bL^\fr,\cE_2)),m_1^{b_1,b_2})$ over $\bMC((\bL^\fr,\cE_1)) \times \bMC((\bL^\fr,\cE_2))$. Note that here we will use a different convention of framing Lagrangians and localized mirror functors than in \cite{HLT24} in order to better align with the works of Nakajima. With additional work, one could also follow the conventions in \cite{HLT24}, and the results would still hold. In this paper, we will take the framing Lagrangian to always have rank 1 and intersect all connected compact components (spheres) in the domain of the Lagrangian immersion. This framing Lagrangian can be constructed as follow: 
 \begin{Construction}
 	Given a graph $D$, we first add an additional vertex such that it connects to all vertices of $D$. Then we perform the plumbing construction of $T^*\bS^2$ according to the new graph, which gives a Liouville manifold $M$ together with the Lagrangian immersion $L$. Let $T \subset M$ be a cotangent fiber at the component corresponding to the newly added vertex. The framed Lagrangian immersion is obtained by a Lagrangian surgery of $L$ with $T$ at the node $L \cap T$. The surgery is carried out locally around the node which does not affect the rest of the Lagrangian. A local model for the surgery is described in \cite[Section 2e]{Sei00}, and a review can be found, for example, in \cite[Section 2.1]{PW19}.
 \end{Construction}
Up to the end of Section 4, we will stick to this notion unless otherwise specified.


Recall that we choose perfect Morse functions on each connect component of $\hat{\bL}$. Thus, there exists a fiberwise dual between $\CF^*((\bL^\fr,\cE_1,b_1),(\bL^\fr,\cE_2,b_2))$ and $\CF^{2-*}((\bL^\fr,\cE_2,b_2),(\bL^\fr,\cE_1,b_1))$ over $\bMC(\bL^{\fr},\cE_1) \times \bMC(\bL^{\fr},\cE_2)$, where $(b_1,b_2) \in \bMC(\bL^{\fr},\cE_1) \times \bMC(\bL^{\fr},\cE_2)$. This duality allows us to analyze the support of Floer cohomology between branes of nearby ranks, leading to the following identification with the Hecke correspondence.

\begin{thm}\label{thm:H}
Assume $D$ has no edge loops. Consider two branes $(\bL^\fr,\cE_j)$ for $j=1,2$ where $\cE_j$ are trivial bundles over $\bL^\fr$ of rank $(\bv^j,\bw)$ for $\bv^2 = \bv^1 + \mathbf{e}^k$ (meaning that $\cE_2$ has one higher rank than $\cE_1$ at the $k$-th component). Then the support of the Floer cohomology $\HF^2((\bL^\fr,\cE_2),(\bL^\fr,\cE_1))$ is a holomorphic Lagrangian subvariety $\cS_k \subset \bMC(\bL^{\fr},\cE_2) \times \bMC(\bL^{\fr},\cE_1)$. Moreover, there exists a change of coordinates on $\bMC(\bL^{\fr},\cE_2) \times \bMC(\bL^{\fr},\cE_1)$ and a change of basis on $\CF((\bL^{\fr},\cE_2),(\bL^{\fr},\cE_1))$ under which $\cS_k$ is identified with the Hecke correspondence $\mathfrak{P}_k(\bv_2,\bw)$ in $\cM(\bv^2,\bw) \times \cM(\bv^1,\bw)$.

Furthermore, $\HF^2((\bL^\fr,\cE_2),(\bL^\fr,\cE_1))$ is a line bundle over the Hecke correspondence $\mathfrak{P}_k(\bv_2,\bw)$.
\end{thm}
\begin{proof}
	In this proof, we work on $\CF^*((\bL^{\fr},\cE_1),(\bL^{\fr},\cE_2))$, which is the fiberwise dual of $\CF^*((\bL^{\fr},\cE_2),(\bL^{\fr},\cE_1))$. 
	
   	Let $b_1:=[B_1,i_1,j_1]$ (resp. $b_2:=[B_2,i_2,j_2]$) be a Maurer-Cartan element of $(\bL^\fr,\cE_1)$ (resp. $(\bL^\fr,\cE_2)$). The Morse model produces the following Lagrangian Floer complex $\CF^*((\bL^\fr,\cE_1),(\bL^\fr,\cE_2))$ over $\bMC(\bL^{\fr},\cE_1) \times \bMC(\bL^{\fr},\cE_2)$, which can be computed as in \cite[Thm. 5.13]{HLT24}: \begin{equation}\label{eq:monad}
  		 L^0((\bL,\cE_1),(\bL,\cE_2)) \oplus L^0(F,F) \xrightarrow{\sigma} E((\bL,\cE_1),(\bL,\cE_2))
  		 \oplus L^1(F,(\bL,\cE_2))
  		 \oplus L^1((\bL,\cE_1), F)  \xrightarrow{\tau} L^2((\bL,\cE_1),(\bL,\cE_2)),
   	\end{equation} where the notations $L(V^1,V^2)$ and $E(V^1,V^2)$ are defined in Remark \ref{rem: tau}, $\sigma(\xi,d):=m_1^{b_1,b_2}(\xi M, d M_{F})=  (B_2 \xi - \xi B_1)  \oplus  (j_2 \xi -d j_1)  \oplus (\xi i_1 -i_2 d)$ and $\tau$ is more complicated before coordinate change. More precisely, we denote $m_1^{b_1,b_2}(\eta X, a I,b J)$ by $\tau(\eta,a,b)$, and  $$\tau(\eta,a,b)= ((1+ \sum_j c_j(B_{\bar{2}}B_{2})^j)B_{\bar{2}} \eta- \eta (1+ \sum_j c_j(B_{\bar{1}}B_{1})^j)B_{\bar{1}})+a(1+\sum_{k}  d_k(j_1i_1)^k )j_1-(1+\sum_{n}  e_n(i_2j_2)^n )i_2 b ,$$ where $c_j$, $d_k$ and $e_n$ are constant, $B_{\bar{1}}, B_{\bar{2}}$ denote the corresponding opposite arrows of $B_1, B_2$. Note that the first differential arises from the unital property of the fundamental classes, whereas the second differential takes the form $\iota^{b_1,b_2}_\xi \partial_{x_a} m_0^{\bb}$, defined as in Proposition \ref{prop: Hom-Ext}. 
   
   Recall that using the Morse models, $m_0^{\bb}$ counts the pearl trajectories whose outputs are the minimum points of the spheres. Because the immersed Lagrangian is exact and the primitive is constant, these pearl trajectories correspond to constant $2n$-gons mapped to the immersed points and evaluated on the stable manifold of the minimum point, contributing the terms $c_j(B_{\bar{1}}B_1)^jB_{\bar{1}} $, and so forth, in the equation.
   
   Moreover, due to the local involution symmetry  (see for example \cite[Lem. 3.4]{HKL23}), the coefficients of $(B_{\bar{2}}B_2)^j$ and $(B_{\bar{1}}B_1)^j$ are the same, which enables us to perform a coordinate change as in \cite[Thm. 5.2]{HLT24}. After changing the coordinate, $\tau$ is simplified to $\tau(\eta,a,b)= B_2 \eta- \eta B_1+aj_1-i_2b$, and $\CF((\bL^{\fr},\cE_1),(\bL^{\fr},\cE_2))$ becomes a complex of vector bundles over $\cM(\bv^1,\bw) \times \cM(\bv^2,\bw)$.

    
    Notice that by the fiberwise dual, the support of the cokernel of $\sigma^\vee$ coincides with the fiberwise kernel of $\sigma.$ Moreover, by Proposition \ref{prop: Hom-Ext}, $\sigma(\xi,d)\mid_{(b_1,b_2)}=0$ if and only if $(\xi,d)$ defines a homomorphism from $b_1$ to $b_2$ as quiver representations.
    Thus $\HF^2((\bL^\fr,\cE_2),(\bL^\fr,\cE_1)) \cong \mathrm{Coker}(\sigma^\vee)$, as a sheaf over $\bM C(\bL^{\fr},\cE_2) \times \bM C(\bL^{\fr},\cE_1)$, has fiber over $(b_2,b_1)$ being the dual space of homomorphisms from $b_1$ to $b_2$. For simplicity, we will work on the dual space.
    
    Observe that,
    with the fixed stability condition, $d \neq 0$ for any nonzero elements $(\xi,d)$ in Kernel of $\sigma$. Otherwise, we have $j_2 \xi =0$, which implies that $ \Im(\xi)$ forms a subrepresentation of $[B_2,i_2,j_2]$ and  that  $\Im(\xi) \subset \Ker(j_2)$. The stability condition implies $\Im(\xi)=0$, which forces $(\xi,d)=(0,0)$.
    
    Furthermore, we also have $\Ker(\xi)=0$ for any nonzero elements $(\xi,d)$ in Kernel of $\sigma$. In other words, any nonzero homomorphism $(\xi,d)$ is injective. Indeed, one can check that $\Ker(\xi)$ is  $B_1$-invariant subspace of $\Ker(j_1)$.
    Hence,  $\Ker(\xi)=0$ by the stability condition.
    
    Thus after a change of coordinate, the support $\cS_k$ is identified with the subsets of pairs $(\tilde{b}_2,\tilde{b}_1) \in \cM(\bv^2,\bw) \times \cM(\bv^1,\bw)$ such that there exist an injective homomorphism $(\xi,d)$ with $d\not=0$ between them. Since $\sigma$ is homogeneous in $(\xi,d)$, this condition is equivalent to the existence of an injective framed homomorphism $(\xi/d,1)$ between them. This coincides with the Hecke correspondence $\mathfrak{P}_k(\bv_2,\bw)$ defined by Nakajima, see Definition 5.6 in \cite{Nak98}, which is a holomorphic Lagrangian subvariety of $\bM (\mathbf{v}^1,\mathbf{w})\times \bM (\mathbf{v}^2,\mathbf{w})$ by Theorem 5.7 in \cite{Nak98}.
    
    Finally, let $(\xi_1,d)$ and $(\xi_2,d)$ be elements in $ \Ker\sigma \mid_{(b_1,b_2)}$ for any $(b_2,b_1)$ in the Hecke correspondence. Then $\sigma(\xi_1-\xi_2,0)$ equals to 0, and hence $\xi_1=\xi_2$ by previous discussion. Therefore, each fiber of  $\HF^2((\bL^\fr,\cE_2),(\bL^\fr,\cE_1))$ is one–dimensional, showing that it forms a line bundle over the Hecke correspondence.
\end{proof}

\begin{cor}
	Assume the same hypotheses as in Theorem~\ref{thm:H}. Let $\mathbf{v}^2= \mathbf{v}^1 +  r \mathbf{e}^k$ with $r \in \Z_{\geq 0}$. Then there exists a change of coordinates on $\bMC(\bL^{\fr},\cE_2) \times \bMC(\bL^{\fr},\cE_1)$ and a change of basis on $\CF((\bL^{\fr},\cE_2),(\bL^{\fr},\cE_1))$ that identifies  the cohomological support of $\HF^2((\bL^\fr,\cE_2),(\bL^\fr,\cE_1))$ with the generalized Hecke correspondence in $\bMC(\bv^2,\bw) \times \bMC(\bv^1,\bw)$. 
	
	Moreover, $\HF^2((\bL^\fr,\cE_2),(\bL^\fr,\cE_1))$ is a line bundle over the generalized Hecke correspondence.
\end{cor}
\begin{proof}
	The computation is analogous to that in Theorem \ref{thm:H}. One can see that after a change of coordinates, the support of $\HF^2((\bL^\fr,\cE_2),(\bL^\fr,\cE_1))$ can be identified with the generalized Hecke correspondence introduced in Section 5.3 of \cite{Nak01}.
\end{proof}

\begin{rem}
	In general, let $L$ be a compact Lagrangian immersion of dimension $n$, 
	and let $\cE_1$ and $\cE_2$ be trivial bundles of ranks $\bv^1 \leq \bv^2$. 
	Suppose there exists a stability condition such that every nonzero morphism 
	between $\cA_L$-representations of dimensions $\bv^1$ and $\bv^2$ is injective. 
	Then the top-degree Floer cohomology $\HF^n((L,\cE_2),(L,\cE_1))$ 
	forms a sheaf over the Hecke-type subvariety $\mathfrak{P}$ in 
	$\bMC(\bv^1) \times \bMC(\bv^2)$. In other words,  $\mathfrak{P} \subset \bMC(\bv^1) \times \bMC(\bv^2)$ is defined to be the subvariety consisting of pairs of $\cA_L$-representations $(V^1, V^2)$, for which there exists an injective homomorphism $\xi:V^1\hookrightarrow V^2$. This follows from the fiberwise dual and Proposition \ref{prop: Hom-Ext}, which can be viewed as a higher-dimensional analogue of Theorem~\ref{thm:H}.
\end{rem}

From the proof, the maximum point of the framing Lagrangian plays an important role in the Hecke correspondence, for which we reformulate as the following corollary. 
\begin{cor} \label{cor:supp}
  Let $\mathbf{v}^1 \leq \mathbf{v}^2$. For any nonzero class in $\HF^2((\bL^\fr,\cE_2),(\bL^\fr,\cE_1))$, every cocycle representative has nonzero coefficient at the maximum point of the framing Lagrangian. 
\end{cor}

Furthermore, we observe that the support of fiberwise cohomology exhibits a structure, induced by the $A_\infty$-product, analogous to the convolution product on Borel–Moore homology. More precisely, for $\mathbf{v}^1 \leq \mathbf{v}^2 \leq \mathbf{v}^3$,  the support of the $A_\infty$-product is $p_{13}(p_{12}^{-1}(Z(\mathbf{v}^1,\mathbf{v}^2))\cap p_{23}^{-1}(Z(\mathbf{v}^2,\mathbf{v}^3)))$, where $p_{kj}$ denotes the projection from $\cM(\bv^1,\bw) \times \cM(\bv^2,\bw) \times \cM(\bv^3,\bw)$ to $\cM(\bv^k,\bw) \times \cM(\bv^j,\bw)$ and $Z(\mathbf{v}^k,\mathbf{v}^j)$ denotes the support of the fiberwise cohomology $\HF^0((\bL^\fr,\cE_k),(\bL^\fr,\cE_j))$. This is because $$m_2^{b_1,b_2,b_3}: \HF^0((\bL^\fr,\cE_1),(\bL^\fr,\cE_2,b_2)) \otimes \HF^0((\bL^\fr,\cE_2,b_2'),(\bL^\fr,\cE_3)) \to \HF^0((\bL^\fr,\cE_1),(\bL^\fr,\cE_3))$$ is zero if $b_2 \neq b_2'$. Hence, the support of the output is contained in  $p_{13}(p_{12}^{-1}(Z(\mathbf{v}^1,\mathbf{v}^2))\cap p_{23}^{-1}(Z(\mathbf{v}^2,\mathbf{v}^3)))$. It remains to show the support is the entire set.

Using the Morse model, the degree zero generators are the maximum points of the irreducible components in the normalization of $\bL^\fr$. Let's denote the maximum points of spheres by $M_{ij} \in \HF^0((\bL^\fr,\cE_i),(\bL^\fr,\cE_j))$ and the framing Lagrangian by $M_{F_{ij}}$ for $i,j=1,2,3$. Given any nonzero elements $\sum \xi_1 M_{12}+ d_1M_{F_{12}}$ and $\sum \xi_2 M_{23}+ d_2M_{F_{23}}$ in $\HF^0((\bL^\fr,\cE_1),(\bL^\fr,\cE_2))$ and $\HF^0((\bL^\fr,\cE_2),(\bL^\fr,\cE_3))$ respectively. Consider the output of $m_2^{b_1,b_2,b_3}(\sum \xi_1 M_{12}+ d_1M_{F_{12}}, \sum \xi_2 M_{23}+ d_2M_{F_{23}})$ at $M_{F_{13}}$. Notice that $M_{F_{ij}}$ is the maximum point of the framing Lagrangian. Thus, the only contribution comes from $m_2^{b_1,b_2,b_3}( d_1M_{F_{12}},  d_2M_{F_{23}})=d_2d_1 M_{F_{13}}$, which is nonzero according to Corollary \ref{cor:supp}.

Hence, the support is $p_{13}(p_{12}^{-1}(Z(\mathbf{v}^1,\mathbf{v}^2))\cap p_{23}^{-1}(Z(\mathbf{v}^2,\mathbf{v}^3)))$, which is analogue to the convolution product on $\bigoplus \limits_{\mathbf{v}^1,\mathbf{v}^2} H_{*}(\mathrm{Supp}\,\HF^0((\bL^\fr,\cE_1),(\bL^\fr,\cE_2))),$ where $H_{*}$ is the Borel-Moore homology.

\subsection{Construction of Holomorphic Lagrangians}
In this subsection, we will construct the holomorphic Lagrangian subvariety introduced by Nakajima in the ADE cases, see for example \cite{Nak94} and \cite{Nak98}, using the Lagrangian Floer theory of the frame Lagrangian branes. More precisely, assume $D$ is a finite graph with no cycles. Nakajima proved that $\pi^{-1}(0)$ is a holomorphic Lagrangian subvariety of $\bM (\mathbf{v},\mathbf{w})$, where $\pi: \bM (\mathbf{v},\mathbf{w}) \to \bM_0 (\mathbf{v},\mathbf{w})$ is the natural projective morphism between the Nakajima quiver varieties, arising from varying the stability conditions.

Let's first reconstruct a filtration on $\bMC (\mathbf{v},\mathbf{w})$ using Lagrangian Floer theory. In some special cases, it gives an irreducible component of the holomorphic Lagrangians in  $\bMC (\mathbf{v},\mathbf{w})$.

\begin{lem}\label{lem:fil}
	The Floer cohomology $\HF^2(\bS_k,(\bL^\fr,\cE))$ gives a filtration of $\bMC (\mathbf{v},\mathbf{w})$ for each irreducible sphere $\bS_k$ in the normalization of $\bL^\fr$.
\end{lem}
\begin{proof}
	Similar to Equation \ref{eq:monad}, one can compute
	\begin{equation}
		\CF^*(\bS_k,(\bL^\fr,\cE)):	0 \xrightarrow{} V_k \xrightarrow{\sigma} \bigoplus_{a:t(a)=k} V_{h(a)} \oplus W_k \xrightarrow{\tau} V_k \xrightarrow{} 0,
	\end{equation} where $\sigma$ and $\tau$ are similar to Equation \ref{eq:monad}, $V_j$ (resp. $W_k$) denotes the tautological bundle associated with the $j-$th vertex (the $k$-th framing vertex) over $\bMC(\bv,\bw)$. 
	
	By stability condition, $\sigma$ is injective; otherwise there exists a proper sub-representation contained in kernel of $j$. In general, the dimension of $\HF^2(\bS_k,(\bL^\fr,\cE))= \mathrm{Coker}\,\tau$ varies fiberwise, which induces a filtration of $\bMC(\bv,\bw)$ according to the dimensions. More precisely, we can define the following subsets of $\bMC(\bv,\bw)$:
	\begin{align}
		\bMC_{k,r}(\bv,\bw):=& \{b_0=[B,i,j] \in \bMC(\bv,\bw) \mid \HF^2(\bS_k, (\bL^\fr,E,b_0))=r\}, \\ 
		\bMC_{k,\leq r}(\bv,\bw):=& \bigcup_{s\leq r} \bMC_{k,s}(\bv,\bw).\notag
	\end{align} In particular, $\bMC_{k,\leq r}(\bv,\bw)$ is an open subset of $\bMC(\bv,\bw)$. 
\end{proof}

In particular, $\bMC_{k,\geq v_k}(\bv,\bw)$ is a closed subvariety. We will use this observation to construct an irreducible component of the holomorphic Lagrangian in the framed $A_n$ case. Let $D$ be the $A_n$ Dynkin diagram. We then add a vertex intersecting the first and last vertices and carry out the construction described in Section \ref{sec:Hecke}, from which we obtain a framed Lagrangian $\bL^\fr$.

\begin{prop}
	Let $\bL^\fr$ be the framed Lagrangian constructed above, equipped with a trivial vector bundle $E$ with rank $(\vec{\delta},1),$ where $\vec{\delta}=(1,\ldots,1)$. Then $$\HF^2(\bS_k,(\bL^\fr,\cE)) \cong \iota_* \mathcal{L},$$  where $\cL$ is a line bundle over the $k$-th exceptional divisor of $\pi: \bMC(\vec{\delta},1) \to \bMC_0(\vec{\delta},1)$ and $\iota$ denotes the inclusion map.
\end{prop}
\begin{proof}
	By construction, the quiver $Q^\fr$ associated to $\bL^\fr$ is the following: $$
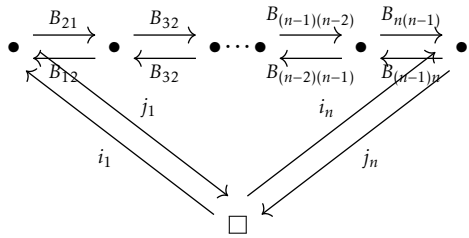
\begin{tikzcd}
		\bullet \arrow[r, shift left=1ex, "B_{21}"] \arrow[dddrr,shift left=1ex,"j_1"] & \bullet \arrow[l, shift left=1ex, "B_{12}"] \arrow[r, shift left=1ex, "B_{32}"] & \bullet \arrow[l, shift left=1ex, "B_{32}"] \cdots \bullet \arrow[r, shift left=1ex, "B_{(n-1)(n-2)}"] & \bullet \arrow[l, shift left=1ex, "B_{(n-2)(n-1)}"]  \arrow[r, shift left=1ex, "B_{n(n-1)}"] & \bullet \arrow[l, shift left=1ex, "B_{(n-1)n}"] \arrow[dddll,shift left=1ex,"j_n"] \\ \\ \\
		&&\boxed{} \arrow[lluuu,shift left=1ex, "i_1" ]  \arrow[rruuu,shift left=1ex, "i_n" ]
	\end{tikzcd}.$$
	
	By Lemma \ref{lem:fil}, we have $$\tau(\eta_{(k-1)k},\eta_{(k+1)k},0)= B_{k(k-1)} \eta_{(k-1)k}+ B_{k(k+1)}\eta_{(k+1)k}$$ for all $(\eta_{(k-1)k},\eta_{(k+1)k},0) \in E(\bS_k, (\bL^\fr,\cE) ) \oplus L(\bS_k, F)$ and $k\neq 1, n$. Moreover, $\tau(0,\eta_{21},b)= B_{12} \eta_{21}+ i_1b$ and $\tau(\eta_{(n-1)n},0,b)= B_{n(n-1)}\eta_{(n-1)n}+ i_nb.$

	Due to the dimension, $\HF^2(\bS_k,(\bL^\fr,\cE))= \mathrm{Coker}(B,i)$ is a torsion sheaf supported at the closed subvariety $\bMC_{k,1}(\vec{\delta},1)$, defined by $B_{k(k-1)}= B_{k(k+1)}=0.$ In particular, the fiber has dimension 1.
	
	On the other hand, the stability condition implies that the only nonzero morphisms are $B_{i(i+1)}$ for $i \leq k-1$, $B_{(t+1)t}$ for $t \geq k$ and $j_1, j_n$. We are interested in the equivalent classes of these morphisms. Using the group actions, all nonzero linear maps except $B_{(k-1)k}$ and $B_{(k+1)k}$ can be normalized to 1 and  $(B_{(k-1)k},B_{(k+1)k}) \in \C^2 \setminus \{0\}.$ Using the $\C^*$-action at the $k$-th vertex, we see that $\bMC_{k,1}(\vec{\delta},1) \cong \CP^1$, which coincides with the $k$-th exceptional divisor in the minimal resolution. Similar arguments apply to the cases $k=1$ and $k=n$.
\end{proof}

Moreover, for general framed Lagrangian branes, we can construct the holomorphic Lagrangians inside the Nakajima quiver variety by taking away several copies of the compact components. This gives a 'dual' construction of the above proposition. 

\begin{example}
	As a concrete illustration, consider the $A_2$ case, and let $\bL^\fr$ be the framed Lagrangian brane of rank $(\bv=(1,1),\bw=(1,1))$. Take $n=k=1$ in the Corollary, we first observe that $\bMC(\bv- \mathbf{e}^1, \bw)$ is a single point. Moreover, $\HF^2((\bL^\fr,\cE),(\bL^\fr,\cE'))= \mathrm{Coker}(\tau)$ is a torsion sheaf supported on a subvariety defined by $B_{12}=i_1=i_2=0.$ This follows from the fact that $\HF^2((\bL^\fr,\cE),(\bL^\fr,\cE'))$ is fiberwise dual to $\HF^0((\bL^\fr,\cE'),(\bL^\fr,\cE))$. Consequently, the only nonzero arrows are $[B_{12},j_1,j_2]$, which defines a holomorphic Lagrangian $\CP^1 \subset \bMC(\bv, \bw).$ Similarly, one can obtain the other holomorphic Lagrangian by $\HF^2((\bL^\fr,\cE),(\bL^\fr,\cE")),$ where $\cE"$ has rank $(\bv-\mathbf{e}^2, \bw).$
\end{example}

This example demonstrates the mechanism in a simple case. More generally, in the ADE setting and for arbitrary dimension vectors, one can construct, instead of irreducible components, the union of holomorphic Lagrangians inside the corresponding Nakajima quiver varieties.

\begin{thm}\label{thm: lag}
	Let $D$ be a graph, $(\bL^\fr,\cE)$ be a stable framed Lagrangian brane of rank $(\mathbf{v},\mathbf{w})$ and $F$ be the framing Lagrangian brane of rank $\mathbf{w}$. Then the support of the degree two Floer cohomology $\HF^2((\bL^\fr,\cE),F)$ is a subvariety in $\bM (\mathbf{v},\mathbf{w})$ defined by $i=0$.
\end{thm}
 \begin{proof}
   	Let's first compute the Floer complex $\CF^*((\bL^\fr,\cE),F)$. As $\bL^\fr$ and $F$ intersect cleanly, using the Morse model, the Floer complex is generated by the minimum point $P_F$ of the framing Lagrangian, which is of degree 2, and the immersed sectors $J$ from the compact Lagrangian spheres to the framing Lagrangian, which have degree 1. Thus, the Floer complex is of the following form: \begin{equation}
  	 \CF^*((\bL^\fr,\cE),F): L^1((\bL,\cE),F) \xrightarrow{\tau}  L^2(F,F),
   	\end{equation} which is fiberwise dual to \begin{equation}
   	\CF^*(F,(\bL^\fr,\cE)):L^0(F,F)  \xrightarrow{\sigma}  L^1(F,(\bL,\cE)).
   \end{equation} Here $\sigma(\xi)=m_1^{b,0}(\xi M_F)=m_2(\xi M_F, \sum i_k I_k)= \sum_k (i_k \xi) I_k,$ since $M_F$ is the fundamental class of the framing Lagrangian. And $\tau(\eta)= m_1^{b,0}(\sum \eta_k J_k)= \sum m_2(i_k I_k, \eta_k J_k)= (\sum_{k} \eta_ki_k ) P_F$. Therefore, $\HF^2((\bL^\fr,\cE),F)$ is a coherent sheaf with the schematic support $$L(\mathbf{v}):=\{b=[(B,i,j)] \in \bM (\mathbf{v},\mathbf{w}) \mid i=0 \}.$$ 
   \end{proof}

\begin{cor} \label{cor: hlag}
	Let $D$ be an ADE Dynkin diagram, $(\bL^\fr,\cE)$ be a stable framed Lagrangian brane of rank $(\mathbf{v},\mathbf{w})$ and $F$ be the framing Lagrangian brane of rank $\mathbf{w}$. Then the support $L(\mathbf{v})$ of $\HF^2((\bL^\fr,\cE),F)$ forms a holomorphic Lagrangian subvariety in $\bM (\mathbf{v},\mathbf{w})$ and $\HF^2((\bL^\fr,\cE),F)$ is the structure sheaf of $L(\mathbf{v})$.
\end{cor}
\begin{proof}
	By above Theorem, for any element $b=[(B,0,j)] \in L(\mathbf{v})$, we have $\sum_{h(a)=v} \epsilon(a)B_a B_{\bar{a}}=0.$ It remains to show $L(\mathbf{v})=\pi^{-1}([0]),$ where $\pi: \bM (\mathbf{v},\mathbf{w}) \to \bM_0 (\mathbf{v},\mathbf{w})$ is the canonical projective morphism. According to Lemma 5.9 in \cite{Nak94}, it suffices to show that for any $b=[(B,0,j)]$, the closure of the orbit $G_{\mathbf{v}} \cdot B$ contains $0$. This is done by the Proposition 14.2 in \cite{Lus91} or Proposition 6.7 in \cite{Nak94}, where the authors showed any representations of the unframed ADE quiver $Q$ such that $\sum_{h(a)=v} \epsilon(a)B_a B_{\bar{a}}=0$ are nilpotent. Hence, the closure of the orbit contains zero. 
	
	Moreover, by the definition of the Maurer-Cartan deformation space, the automorphism groups of the trivial vector bundles over $F$ remain in the construction. Therefore, $\HF^2((\bL^\fr,\cE),F)$ is the structure sheaf, or the trivial holomorphic line bundle over $L(\mathbf{v})$.
\end{proof}

Without the assumption that $D$ is an ADE Dynkin diagram, the subvariety defined by $i=0$ is not a holomorphic Lagrangian.
\begin{example}\label{ex: cexample}
	Recall that the Nakajima quiver variety $\cM(\bv,\bw)$ has dimension $\bv^{t}(2\bw-C \bv)$, where $C$ is the Cartan matrix of the graph $D$, see for example Equation (4.6) in \cite{Nak94}.
	
	In particular, take $D$ to be an affine $A_1$ graph and $\bv=\bw=(1,1)$. Then $\cM(\bv,\bw)$ has dimension 4. However, the subvariety defined by $i_1=i_2=0$ only has dimension 3. Indeed, for $\bv=(1,1)$, the moment map condition $B_{\bar{1}}B_1=B_{\bar{2}}B_2$ is equivalent to $B_1B_{\bar{1}}=B_2B_{\bar{2}}$ if $i_1=i_2=0$. Thus, by dimension count, the subvariety defined by $i=0$ is not a holormophic Lagrangian.
	
	A similar phenomenon occurs in the presence of edge loops. For example, let $D$ be the affine $A_0$ graph, and take $\bv=\bw=1$. In this case, $(B_1,B_{\bar{1}})$ are complex numbers. Thus $[B_1,B_{\bar{1}}]=0$. The moment map equation implies that $ij=0$. However, stability requires $j \neq 0$. Hence, $\cM(1,1)$ is isomorphic to the subvariety defined by $i=0$, which is not a holomorphic Lagrangian.
\end{example}

\subsection{Structural Analogy with Hecke Correspondence}
While we do not construct creation and annihilation operators in this work, we observe that certain geometric operations in Lagrangian Floer theory exhibit structural similarities with Hecke correspondences.

In the previous section, we observed that the support of the fiberwise cohomology $\HF^0((\bL^\fr,\cE_k),(\bL^\fr,\cE_j))$ carries a structure analogous to the convolution product on Borel-Moore homology. In this subsection, we focus on the case that rank $\cE_k=(0,\bw)$, and the fiberwise support is a holomorphic Lagrangian $L(\bv_j)$, as shown in Corollary \ref{cor: hlag}.

We begin with the following lemma, which is useful for our later discussions.

\begin{lem}
	If the rank of $\cE_1$ is less than or equal to the rank of $\cE_2$, then the $A_\infty$-product $$m_2^{0,b_1,b_2}: \HF^0(F,(\bL^\fr,\cE_1)) \otimes \HF^0((\bL^\fr,\cE_1),(\bL^\fr,\cE_2)) \to \HF^0(F,(\bL^\fr,\cE_2))$$ is fiberwise injective.
\end{lem}
\begin{proof}
	Let $\sum \xi_1 M_{12}+ d_1M_{F_{12}}$ be a nonzero element in $\HF^0((\bL^\fr,\cE_1),(\bL^\fr,\cE_2))$. Here we use the same notations as in Section \ref{sec:Hecke}. As shown in the proof of Theorem \ref{thm:H}, $\HF^0(F,(\bL^\fr,\cE_1))$ is generated by the maximum point of the framing Lagrangian $M_{F_1}$.
	
	Since $M_{F_1}$ is the fundamental class of the framing Lagrangian, we have $m_2^{0,b_1,b_2}(M_{F_1},\sum \xi_1 M_{12}+ d_1M_{F_{12}})=m_2^{0,b_1,b_2}(M_{F_1},d_1 M_{F_{12}})=d_1 M_{F_2},$  Thus,  $m_2^{0,b_1,b_2}$ is injective.
\end{proof}

Since $m_2^{0,b_1,b_2}$ is injective, for any subsheaf $\cF$ of $\HF^0(F,(\bL^\fr,\cE_1))$, the $A_\infty$-product $m_2^{0,b_1,b_2}$ would induce a correspondence on the support of $\cF$. 

Now, let $\mathbf{v}^2= \mathbf{v}^1 +  \mathbf{e}^k$ as before. For simplicity, let's denote the fiberwise cohomology $\HF^0((\bL^\fr,\cE_1),(\bL^\fr,\cE_2))$ by $\cF^k_{\mathbf{v}^2}$ with support lying in the Hecke correspondence $\cS_k(\mathbf{v}^1,\mathbf{w})$. Consider the action of the formal sum of $\cF^k_{\mathbf{v}^2}$ on $\bigoplus \limits_{\mathbf{v}} \HF^0(F,(\bL^\fr,\cE))$ via $m_2^{0,\bb,\bb}$ for $\forall k$:
\begin{equation}
	F_k= \sum \limits_{\mathbf{v}^2} \cF^k_{\mathbf{v}^2} \bullet : \bigoplus \limits_{\mathbf{v}}  \HF^0(F,(\bL^\fr,\cE)) \to \bigoplus \limits_{\mathbf{v}}  \HF^0(F,(\bL^\fr,\cE)).
\end{equation}
In other words, for any $\alpha \in \bigoplus \limits_{\mathbf{v}} \HF^0(F,(\bL^\fr,\cE)),$ we define $F_k(\alpha):= m_2^{0,\bb,\bb}(\alpha, \beta) \in \bigoplus \limits_{\mathbf{v}} \HF^0(F,(\bL^\fr,\cE))$ for $\forall \beta \in \cF^k_{\mathbf{v}^2}.$ In particular, since $m_2^{0,\bb_1,\bb_2}$ is injective, the support of $$m_2^{0,\bb_1,\bb_2}: \HF^0(F,(\bL^\fr,\cE_1)) \otimes \HF^0((\bL^\fr,\cE_1),(\bL^\fr,\cE_2)) \to \HF^0(F,(\bL^\fr,\cE_2))$$ is $p_2(p_1^{-1}(L(\mathbf{v}^1))\cap \cS_k(\mathbf{v}^1,\mathbf{w}))$, which closely resembles the structure of the creation operator $F_k$.

Similarly, we define the left action of the formal sum of $\cF^k_{\mathbf{v}^2}$ on $\bigoplus \limits_{\mathbf{v}} \HF^0(F,(\bL^\fr,\cE))$ and then take its dual:
\begin{equation}
	E_k= (\sum \limits_{\mathbf{v}^2} \cF^k_{\mathbf{v}^2} (\bullet)^\lor)^\lor : \bigoplus \limits_{\mathbf{v}} \HF^0(F,(\bL^\fr,\cE)) \to \bigoplus \limits_{\mathbf{v}} \HF^0(F,(\bL^\fr,\cE)).
\end{equation}
More precisely, for any $\alpha \in \bigoplus \limits_{\mathbf{v}} \HF^0(F,(\bL^\fr,\cE)),$ we define $E_k(\alpha):= (m_2(\beta,(\alpha)^\lor))^\lor \in \bigoplus \limits_{\mathbf{v}} \HF^0(F,(\bL^\fr,\cE))$ for $\forall \beta \in \cF^k_{\mathbf{v}^2}.$ Here, the dual $(\alpha)^\lor$ is the dual element in $\HF^2((\bL^\fr,\cE),F)$. In particular, the support of $$m_2^{0,\bb_1,\bb_2}: \HF^0((\bL^\fr,\cE_1),(\bL^\fr,\cE_2)) \otimes \HF^2((\bL^\fr,\cE_2),F) \to \HF^2((\bL^\fr,\cE_1),F)$$ is $p_1(p_2^{-1}(L(\mathbf{v}^1))\cap \cS_k(\mathbf{v}^1,\mathbf{w}))$, which closely resembles the structure of the annihilation operator $E_k$.

Besides, we also have the identity operator on $\HF^0(F,(\bL^\fr,\cE))$, which is given by $1_{\mathbf{v}}=\sum M_{11}+M_{F_{11}} \in \HF^0((\bL^\fr,\cE),(\bL^\fr,\cE))$. More precisely, we define:
\begin{equation}
	H_{\mathbf{v}}= m_2(\bullet,1_{\mathbf{v}}) : \HF^0(F,(\bL^\fr,\cE_1)) \to \HF^0(F,(\bL^\fr,\cE)).
\end{equation}
Sometimes we will also write $H_{\mathbf{v}}$ as $\Delta(\mathbf{v},\mathbf{w}).$

In geometric representation theory, Nakajima introduced the creation and annihilation operators on the Borel-Moore homology $\bigoplus \limits_{\mathbf{v}} H_{\mathrm{top}}(\mathrm{Supp}\,\HF^0(F,(\bL^\fr,\cE)))$ via Hecke correspondence. As a result, this direct sum forms an integrable highest weight module over the deformed Kac-Moody algebra, see Theorem 10.2 in \cite{Nak98}. 

The operators constructed above are structurally analogous to those defined by Nakajima. The above discussion does not provide a literal identification with Hecke correspondences, but rather suggests that Floer theory captures a shadow of the representation-theoretic structures encoded in Nakajima’s work.

\subsection{Modified Quiver Varieties}
In this subsection, we show that the Maurer–Cartan deformation spaces of certain framed Lagrangian branes can be identified with modified quiver varieties. Moreover, for these Lagrangian branes, the construction of Hecke correspondences for the modified quiver varieties also applies.

To make this identification explicit, we now describe the construction of these modified framed Lagrangian branes. Let's first take a framed Lagrangian brane $(\bL^\fr,\cE)$ in the plumbing space as in Section \ref{sec:Hecke}. Recall that each compact connected component in the domain of $\bL^\fr$ is a sphere. In particular, each compact component admits a perfect Morse function with a unique maximum point and a unique minimum point. Let $\widetilde{\bL}_k^\fr=\bL^\fr \setminus \{p_i\}_{i \neq k \in I},$ where $p_i$ is the minimum point of the perfect Morse function on the $i$-th sphere $\bS^2$. Furthermore, let's fixed a basis of $\cE$ over each sphere except the $k$-th sphere. Let's write the modified framed Lagrangian brane as $(\widetilde{\bL}_k^\fr,\cE).$ In the following, we will restrict to the Fukaya subcategory generated by $\widetilde{\bL}_k^\fr$ and take the stablity condition considered above.

\begin{prop}
	Up to a change of coordinates, the Maurer-Cartan space of $(\widetilde{\bL}_k^\fr,\cE)$ is the modification of the quiver variety introduced by Nakajima, $$\bM C(\widetilde{\bL}_k^{\fr},\cE)=\mu_k^{-1}(0)^s/GL(V_k),$$ where $\mu_k^{-1}(0)$ denotes the $k$-th component (i.e. $\Hom(V_k,V_k)$ component) of the obstruction equation $\mu^{-1}(0)$ of $(\bL^\fr,\cE)$ and $GL(V_k)$ is the automorphism group at the $k$-th vertex. Furthermore, let $\mathbf{v}^2= \mathbf{v}^1 + n \mathbf{e}^k$ for $n \in \Z_{\geq 0}$. Then the cohomological support of the Floer cohomology $\HF^2((\widetilde{\bL}_k^\fr,\cE_2),(\widetilde{\bL}_k^\fr,\cE_1))$ is a modification of the Hecke correspondence.
\end{prop}
\begin{proof}
	Recall that in \cite{HLT24}, the authors showed that the obstruction equation $m_0^{b}=0$ of $(\bL^\fr,\cE)$ coincides, up to a change of coordinates, with the level-zero condition for the complex moment map $\mu$ defining the Nakajima quiver variety. More precisely, the obstruction term $m_0^b$ counts pearl trajectories with outputs at the minimum point of each sphere,  and the obstruction equations arise from the coefficients of these outputs. Since $\bL^\fr$ is exact with constant primitives, $m_0^b$ counts the constant discs as in \cite{HLT24}. Besides, by the construction of $\widetilde{\bL}_k^\fr$, the minimum point $P_i$ of the sphere $\bS^2_i$ is removed for $i\neq k$. Therefore, $m_0^b$ only has outputs at $P_k$, which precisely recovers the modified moment map equation at the 
	$k$-th vertex $\mu_k$. This completes the proof of the first part of the theorem.
	
	For the latter statement, we consider the fiberwise dual $\CF^*((\widetilde{\bL}_k^\fr,\cE_1),(\widetilde{\bL}_k^\fr,\cE_2))$ of $\CF^*((\widetilde{\bL}_k^\fr,\cE_2),(\widetilde{\bL}_k^\fr,\cE_1))$ as in the proof of Theorem \ref{thm:H}. Besides, by the construction of the modified framed Lagrangian branes, $\CF^*((\widetilde{\bL}_k^\fr,\cE_1),(\widetilde{\bL}_k^\fr,\cE_2))$ agrees with Equation \ref{eq:monad}, except the second differential $\tau$ and the final term. More precisely, we have the following complex:
	\begin{equation}
		L^0((\bL,\cE_1),(\bL,\cE_2)) \oplus L^0(F,F) \xrightarrow{\sigma} E((\bL,\cE_1),(\bL,\cE_2))
		\oplus L^1(F,(\bL,\cE_2))
		\oplus L^1((\bL,\cE_1), F)  \xrightarrow{\tau \mid_k} L^2((\bS_k,\cE_1),(\bS_k,\cE_2)),
	\end{equation} where $\bS_k$ is the $k$-th sphere and $\tau \mid_k$ is the restriction of $\tau$ to the component of immersed sectors with head or tail equal to $k$. Consequently, the result follows from this complex using arguments analogous to those in the proof of Theorem \ref{thm:H}.
\end{proof}	
	
	
\begin{rem}
	In \cite{Nak01}, the modified quiver variety $\bM C(\widetilde{\bL}_k^{\fr},\cE)$ is denoted as $\widetilde{\cM}(\mathbf{v},\mathbf{w})^{\circ}$, which is an open subvariety of $\widetilde{\cM}(\mathbf{v},\mathbf{w})$. The latter space, $\widetilde{\cM}(\mathbf{v},\mathbf{w})$, is the space of isomorphism classes of Maurer-Cartan elements $b$ of $(\bL_k^{\fr},\cE)$ such that $\HF^0(\bS_k,(\bL_k^{\fr},\cE,b)) = 0$. Note that in a neighborhood of the $k$-th sphere, the immersed Lagrangian is homeomorphic to the union of the sphere and several cotangent fibers. Along these lines, one can observe that $\widetilde{\cM}(\mathbf{v},\mathbf{w})$ is isomorphic to the product of the Nakajima quiver variety for the $A_1$ graph and an affine space, thereby recovering Nakajima's original observation.
\end{rem}
	
	Similarly, one can consider the Maurer-Cartan space of $\widetilde{\bL}_J^\fr=\bL^\fr \setminus \{p_i\}_{i \in I \setminus J},$ where $J$ is a subset of the set of vertices $I$. The analysis proceeds in a similar manner as in the previous case, and we omit the details of the proof.
	
	\begin{cor}
		Up to a change of coordinates, the Maurer-Cartan space of $(\widetilde{\bL}_J^\fr ,\cE)$ is $$\bM C(\widetilde{\bL}_J^{\fr},\cE)=\mu_J^{-1}(0)^s/GL_J,$$ where $\mu_J^{-1}(0)$ denotes the corresponding component of the obstruction equation $\mu^{-1}(0)$ of $(\bL^\fr,\cE),$ and $GL_J$ is the corresponding automorphism group.
	\end{cor}

\section{Local Homological Mirror Symmetry for ADE surfaces}\label{sec: CY}

\subsection{Stability Conditions and Lagrangian Floer Theory}

In the remaining sections, we will restrict the localized mirror functor to the subcategory that has no convergence issues and work over the complex number as before. 

Let $M$ be the symplectic manifold obtained by the plumbing the cotangent bundles $T^*\bS^2$ according to the graph $D$, and let $L$ (resp. $\bL^\fr$) be a (resp. framed) Lagrangian immersion, defined as the union of the zero sections (resp. together with one cotangent fiber over each component).
\begin{lem}
	Let $(\bL^\fr,E)$ denote the stable framed Lagrangian brane of rank $(\mathbf{v},\mathbf{w})$ with the deformation parameter $b:=\sum B_{a}X_a + \sum i_{v} I_v + \sum j_{v} J_v $, and let $G_k$ be a cotangent fiber at the $k-$th sphere other than the framing Lagrangian $F_k$. Then $\cF^{(\bL^\fr,E)}(G_k)$ is the tautological bundle over $\cM(\bv,\bw)$ corresponding to the $k$-th vertex, of rank $\bv_k$.
\end{lem}
\begin{proof}
	Observe that the cotangent fiber $G_k$ intersects with $\bL^\fr$ transversally at a point $P$ at the $k$-th sphere. Hence, by definition, for any fixed $b_0 \in \cM(\bv,\bw)$, $\cF^{(\bL^\fr,\cE,b_0)}(G_k)= E_k\mid_P \langle P \rangle$. Therefore, $\cF^{(\bL^\fr,\cE,b_0)}(G_k)\cong \cE_k,$ where $\cE_k$ is precisely the tautological bundle associated with
	$k-$th vertex. In other words, the fiber $\cE_k$ at $b_0 \in \bMC(\bv,\bw)$ is given by the $GL_{\bv_k}$ representation $V_k$ assigned to the $k$-th vertex.
\end{proof}

\begin{rem}
	If we consider the noncommutative deformation space of $\bL^\fr$, then by definition, $\cF^{(\bL^\fr,\bb)}(G_k)$ is the projective module $\cA_{\bL^\fr} \cdot e_k$, where $e_k$ denotes the $k$-th idempotent.
\end{rem}

In particular, when $D$ is an affine ADE Dynkin diagram, we have the following description of the stability condition of the framed quiver representations, see Definition \ref{def: stab}. Here $\zeta^0_{\R}\in\mathbb{R}^{Q_0}$ denotes a generic stability parameter in the sense of Nakajima, while $\zeta_{\R}$ lies in the corresponding chamber (see Section 2.1 in \cite{Nak07} for details). This is not a new criterion, but rather the mirror translation of Nakajima’s notion of stability, providing a geometric realization of the stable moduli locus via the Maurer–Cartan deformation space.

\begin{lem}[Proposition 4.1 in \cite{Nak07}] \label{lem: stab}
	Let $D$ be an affine ADE Dynkin diagram, and $b_0=\sum B_{a}X_a + \sum i_{v} I_v + \sum j_{v} J_v$ be a solution of the Maurer-Cartan equation of $(\bL^\fr,E')$, i.e. $m_0^{b_0}=0$. Then 
	\begin{enumerate}
		\item $b_0$ is $\zeta_\R$-stable as a quiver representation if and only if $\HF^0((L,E,b),(\bL^\fr,E',b_0))$ vanishes except finitely many $b$ and $\HF^2((L,E,b),(\bL^\fr,E',b_0))=0$ for all $b \in \bMC(L,E)$.
		\item $b_0$ is $\zeta^0_\R$-stable as a quiver representation if and only if $$\HF^0((L,E,b),(\bL^\fr,E',b_0))=\HF^2((L,E,b),(\bL^\fr,E',b_0))=0$$ for all but finitely many $b\in \bMC(L,E)$.
		\item $b_0$ is $\zeta^0_\R$-stable as a quiver representation if and only if $$\HF^0((L,E,b),(\bL^\fr,E',b_0))=\HF^2((L,E,b),(\bL^\fr,E',b_0))=0$$ for all $b\in \bMC(L,E)$.
	\end{enumerate}
\end{lem}
\begin{proof}
	By Theorem 5.13 of \cite{HLT24}, the Floer complex 
	\(\left(\CF^*((L,E,b),(\bL^{\fr},E',b_0)), m_1^{b,b_0}\right)\)
	is identified with the monadic complex, which is quasi-isomorphic to the (framed) torsion-free sheaf over the (compactification of) ALE space.
	Under this identification, the Lagrangian Floer cohomology $\HF^0$ and $\HF^2$ are the obstructions to be stable, which follows from Proposition 4.1 in \cite{Nak07}. 
\end{proof}

\begin{rem}
	Recall that if $D$ is an affine ADE Dynkin diagram, the derived McKay correspondence \cite{KV00, BKR01} shows that
	$$\mathbb{R}\mathrm{Hom}(\mathcal{T},-):D^b\mathrm{Coh}(\widetilde{\C^2/\Gamma}) \to D^b(\mathrm{mod}-\cA_\bL)$$ induces an equivalence between the derived categories, where $\Gamma$ is the finite subgroup of $SL_2(\C)$ corresponding to the graph $D$, $\widetilde{\C^2/\Gamma}$ is the crepant resolution of $\C^2/\Gamma$ and $\mathcal{T}$ is the tilting bundle.
	
	Under this functor, one can show that similar results in Lemma \ref{lem: stab} hold by replacing the Lagrangian brane $(L,E,b)$ by the universal brane $(\bL,\bb).$ 
\end{rem}

\subsection{Local Homological Mirror Symmetry}

Homological Mirror Symmetry proposed by Kontsevich \cite{Kon95} conjectures a derived equivalence 
$$\Fuk(M) \simeq D^b\mathrm{Coh}(\check{M})$$
associated to a mirror pair $(M, \check{M})$. 
In this subsection, we consider a localized version of this equivalence. On the symplectic side, we take the Fukaya subcategory generated by $\bL$. On the  mirror side, we consider the Maurer-Cartan algebra $\cA_\bL$. Using the localized mirror functor \cite{CHL17}, we would like to show that it gives a quasi-isomorphism of derived categories. As an application, this framework allows us to establish the mirror correspondence of the stablity condition for ADE surfaces (Corollary \ref{cor: T}). 

\begin{thm}\label{thm:qiso}
	Let $M$ be the plumbing of cotangent bundles $T^*\bS^2$ according to a non-Dynkin diagram $D$, and $\bL \subset M$ be the zero section. Then $\cF^{(\bL,\bb)}(\bL,\bb)$ gives a projective resolution of the preprojective algebra $\cA_\bL$ as an $\cA_\bL$-bimodule. Furthermore, $\cF^{(\bL,\bb)}$ induces a quasi-isomorphism $$\cF^{(\bL,\bb)}:\CF^*((\bL,\bb),U) \to \Hom_{\cA_\bL}(\cF^{(\bL,\bb)}(\bL,\bb), \cF^{(\bL,\bb)}(U)).$$ 
	
	In particular, the localized mirror functor $$\cF^{(\bL,\bb)}:\Fuk^{sub}(M) \to \mathrm{Perf}(\cA_\bL)$$ is a quasi-equivalence, where $\Fuk^{sub}(M)$ denotes the Fukaya subcategory split generated by $(\bL,\bb)$ and $\mathrm{Perf}(\cA_\bL)$ is the dg category of perfect $\cA_\bL$-modules.
\end{thm}
\begin{proof}
	For simplicity, we will omit the sign, which depends on the choice of the spin structure and orientations of $\bL$.
	
	As in Equation \ref{eq:monad}, up to coordinate-change, the Morse model produces the following complex: \begin{equation}\label{eq:resolution}
		(\cF^{(\bL,\bb)}(\bL,\bb), m_1^{\bb,\bb}):
		L((\bL,\bb),(\bL,\bb))  \xrightarrow{\sigma} E((\bL,\bb),(\bL,\bb))  \xrightarrow{\tau} L((\bL,\bb),(\bL,\bb)).
	\end{equation} In the above, $\sigma(\xi_v)=m_1^{\bb,\bb}(\xi_v M_v)=\sum_{ h(a)=v} \left( m_2(x_aX_a,\xi_v M_v) + m_2(\xi_v M_v,x_{\bar{a}}X_{\bar{a}}) \right) =\bigoplus_{h(a)=v} \left(  \xi_v x_a \oplus x_{\bar{a}}\xi_v \right) $, since $M_v$ is the fundamental class of $S^2_v$. Besides, $\tau(\eta_a)= m_1^{\bb,\bb}(\eta_a X_a)= x_{\bar{a}} \eta_a \oplus \eta_a x_{\bar{a}}$. In particular, $\sigma$ corresponds to pre-compose (resp. post-compose) with the arrows whose heads (resp. tail) are at the vertex $v$. 

    Firstly, we show that $\cF^{(\bL,\bb)}(\bL,\bb)$ gives a projective bimodule resolution of $\cA_\bL$, by identifying it with the complex introduced in Proposition 2.5 in \cite{CBK22}. The key step is to identify the differential. 
    
    Let's take $P_0=\oplus_i \cA_\bL e_i \otimes e_i \cA_\bL$ and $P_1= \oplus_a \cA_\bL e_{h(a)}\otimes e_{t(a)}\cA_\bL.$ Recall that in Theorem 2.7 of \cite{CBK22}, the author proved that the following complex is a projective resolution of $\cA_\bL$ as an $\cA_\bL$ bimodule if $D$ is a non-Dynkin diagram:
    $$0 \to \Hom_{\cA_\bL-\cA_\bL}(P_0,\cA_\bL \otimes \cA_\bL) \xrightarrow{-g^*} \Hom_{\cA_\bL-\cA_\bL}(P_1,\cA_\bL \otimes \cA_\bL) \xrightarrow{f^*} \Hom_{\cA_\bL-\cA_\bL}(P_0,\cA_\bL \otimes \cA_\bL) \to 0,$$ where $f(e_i \otimes e_i)= \sum_{h(a)=i} (e_{h(a)} \otimes \bar{a} + a \otimes e_{h(a)})$ and $g(e_{h(a)} \otimes e_{t(a)})=a \otimes e_{t(a)} -e_{h(a)}\otimes a$.

    Let $\varphi_{e_i \otimes e_i} \in \Hom_{\cA_\bL-\cA_\bL}(P_0, \cA_\bL \otimes \cA_\bL)$ be the morphism corresponding to $e_i \otimes e_i \in e_i \cA_\bL \otimes \cA_\bL e_i$. Then $g^*\varphi_{e_i \otimes e_i} (e_{h(a)}\otimes e_{t(a)})= \varphi_{e_i \otimes e_i}(g(e_{h(a)}\otimes e_{t(a)}))=\varphi_{e_i \otimes e_i}(a\otimes e_{t(a)}-e_{h(a)}\otimes a)=ae_i \otimes e_i e_{t(a)}- e_{h(a)}e_i \otimes e_i a$, where $e_i e_j = \delta_{ij}$. In particular, if $t(a)=i$ and $h(a)\neq i$, we have $g^*\varphi_{e_i \otimes e_i} (e_{h(a)}\otimes e_{t(a)})=a \otimes e_i$. If $h(a)=i$ and $t(a)\neq i$, we have $g^*\varphi_{e_i \otimes e_i} (e_{h(a)}\otimes e_{t(a)})=- e_i \otimes a$. Hence, $g^*$ is pre-composing (resp. post-composing)  the arrows heading (resp. tail) at $i$-th vertex, which coincides with $\sigma$. 
    
    Similarly, we can show $f^*$ is the same as $\tau$. Let $\phi_{x}$ be the morphism in $\Hom_{\cA_\bL-\cA_\bL}(P_1, \cA_\bL \otimes \cA_\bL)$ corresponding to $e_{h(x)} \otimes e_{t(x)} \in \bigoplus_{a \in Q^\fr} e_{h(a)} \cA_\bL \otimes \cA_\bL e_{t(a)}$. In particular, the restriction of $\phi_x$ to the component $ \cA_\bL e_{h(a)} \otimes e_{t(a)}  \cA_\bL$ equals zero if $x \neq a$. As before, we can compute $f^*\phi_x(e_i \otimes e_i)=\phi_x( \sum_{h(a)=i} (e_h(a) \otimes \bar{a}+ a \otimes e_{h(a)}) )= e_ie_{h(x)} \otimes \bar{x} + \bar{x} \otimes e_i e_{t(x)}, $ which coincides with $\tau$. Therefore, $\cF^{(\bL,\bb)}(\bL,\bb)$ is a projective resolution of $\cA_\bL$ by Proposition 2.4 and Theorem 2.7 in \cite{CBK22}.
    
    In the end, we will show $$\cF^{(\bL,\bb)}:\CF^*((\bL,\bb),U):=\cA_\bL \otimes_{\mathbb{K}}\CF^*(\bL,U) \to \Hom_{\cA_\bL}(\cF^{(\bL,\bb)}(\bL,\bb), \cF^{(\bL,\bb)}(U))$$ is a quasi-isomorphism, where $\mathbb{K}$ is the semisimple ring formed by the trivial paths. To do so, we use the spectral sequence of the double complex to compute $H^k(\Hom_{\cA_\bL}(\cF^{(\bL,\bb)}(\bL,\bb), \cF^{(\bL,\bb)}(U))).$ For simplicity, let's denote $\cF^{(\bL,\bb)}(\bL,\bb)$ by $A^\bullet$ and $\cF^{(\bL,\bb)}(U)$ by $B^\bullet$. Notice that, the complex $\Hom_{\cA_\bL}^{\bullet}(A^\bullet, B^\bullet)$ is the total complex of the double complex $K^{i,j}:=\Hom(A^{-i},B^j)$ endowed with the two differentials $d_I= \pm d_A$ and $d_{II}=d_B$. In addition, there exists a spectral sequence
    $$E^{p,q}_2:= H^p_{II}H^q_I(K^{\bullet,\bullet}) \Rightarrow H^{p+q}(K^{\bullet,\bullet}).$$
    
    Notice that $E^{p,q}_2:= H^p_{II}H^q_I(K^{\bullet,\bullet}) = 0$ if $q \neq 0.$ If $q=0$, $$E_2^{p,0}=H^p (\Hom_{\cA_\bL}(H^0(A^\bullet), B^\bullet))=H^p(\Hom_{\cA_\bL}(\cA_\bL, B^\bullet))=H^p(\cF^{(\bL,\bb)}(U)).$$ Therefore, we know $$H^p(\cF^{(\bL,\bb)}(U)) \cong H^p (\Hom_{\cA_\bL}(\cF^{(\bL,\bb)}(\bL,\bb), \cF^{(\bL,\bb)}(U))).$$ Moreover, this can be induced by $\cF^{(\bL,\bb)}: \alpha \mapsto m_2^{\bb,\bb,0}(-,\alpha),$ since $m_2^{\bb,\bb,0}(1_{\bL},\alpha)= \alpha$ for $\forall \alpha \in \CF^*((\bL,\bb),U)$. Hence, the first statement holds.
    
    In particular, to obtain the fully-faithfulness $\cF^{(\bL,\bb)}$, it suffices to consider the split generators by Theorem 4.2 in \cite{CHL17}. Therefore, as $\cA_{\bL}$ is quasi-isomorphic to the $\cF^{(\bL,\bb)}(\bL,\bb)=\CF^*((\bL,\bb),(\bL,\bb))$, the last statement follows by taking $U=(\bL,\bb).$
\end{proof}

As a consequence, we have the following mirror symmetry correspondence of Lemma \ref{lem: stab}, the quiver stability condition.

\begin{cor}\label{cor: T}
	For the ADE surfaces, $\mathrm{Ext}^*(\mathcal{T},\cE[-1]) \cong \HF^*((\bL,\bb),(\bL^\fr,b_0)),$ where $\mathcal{T}$ is the tilting bundle and $\cE$ is the torsion-free sheaf over $\widetilde{\C^2/\Gamma}$. In particular, $\mathrm{Ext}^1(\mathcal{T},\cE)=0$ and $\mathrm{Ext}^2(\mathcal{T},\cE)=0$.
\end{cor}
\begin{proof}
	By Theorem \ref{thm:qiso}, we know $\cF^{(\bL,\bb)}$ induces a quasi-isomorphism between $\CF^*((\bL,\bb),U)$ and $\Hom_{\cA_\bL}(\cF^{(\bL,\bb)}(\bL,\bb), \cF^{(\bL,\bb)}(U))$ for affine ADE Dynkin diagram and all Lagrangians $U$. In particular, we can take $U=(\bL^\fr,b_0)$. Under the derived McKay correspondence, $\cF^{(\bL,\bb)}(\bL,\bb)$ corresponds to the tilting bundle $\mathcal{T}$, while $\cF^{(\bL,\bb)}(\bL^\fr,b_0)$ is quasi-isomorphic to the torsion-free sheaf $\cE[-1]$. Hence, we have $$\mathrm{Ext}^*(\mathcal{T},\cE[-1]) \cong \HF^*((\bL,\bb),(\bL^\fr,b_0)).$$
	
	Notice that only $\HF^1((\bL,\bb),(\bL^\fr,b_0))\neq 0$ if $b_0$ is a stable representation of $\cA_\bL^\fr$, see for example Proposition 4.1 in \cite{Nak07} or Lemma \ref{lem: stab}. Hence, the statement follows. 
\end{proof}

As one may observe, the proof of the fully-faithfulness of $\cF^{(\bL,\bb)}$ only relies on the fact that $\cF^{(\bL,\bb)}(\bL,\bb)$ provides a projective bimodule resolution of $\cA_\bL$. This observation immediately leads to the following corollary:

\begin{cor}\label{cor: res}
	Let $\bL$ be a Lagrangian immersion with formal deformation space $\cA_{\bL}$. If $\cF^{(\bL,\bb)}(\bL,\bb)$ gives a projective bimodule resolution of $\cA_\bL$, then the localized mirror functor $$\cF^{(\bL,\bb)}:\Fuk^{sub}(M) \to \mathrm{Perf}(\cA_\bL)$$ is a quasi-equivalence, where $\Fuk^{sub}(M)$ denotes the Fukaya subcategory split generated by $(\bL,\bb)$.
\end{cor}

It would be interesting to extend this local equivalence to global mirrors.



Moreover, the preceding statements appear to be closely intertwined with the theory of bimodule resolutions of Calabi-Yau algebras, when the immersed Lagrangian $\bL$ is compact. 
To make this connection precise, we recall the definitions of homological smoothness and the Calabi–Yau property for associative algebras. See for instance \cite{Gin07,Boc08,VdB15,LamPhD} for the theory of Calabi-Yau algebras.
\begin{defn}
	An associative algebra $A$ is homologically smooth if $A$ is a perfect $A$-bimodule. In other words, $A$, viewed as a bimodule over itself, admits a finite-length resolution by finitely generated projective bimodules.
	
	An associative algebra $A$ is Calabi-Yau of dimension $d$ if $A$ is homologically smooth and there exists an isomorphism $\eta: \mathbb{R}\Hom_{A\otimes A^{op}}(A,A\otimes A^{op}) \to A[-d]$ in the derived category of $A$-bimodule.
\end{defn}

In general, if we takes degree-one boundary deformations in place of flat connections, we have the following proposition concerning the Calabi-Yauness of the formal deformation space $\cA_\bL$ and the localized mirror functor.

\begin{lem}
	Let $\bL$ be a compact, relatively spin, oriented immersed Lagrangian in a symplectic manifold $M$. If $\cF^{(\bL,\bb)}(\bL,\bb)$ provides a projective resolution of $\cA_\bL$ as an $\cA_\bL$-bimodule, then $\cA_\bL$ is a Calabi–Yau algebra.
	
	Moreover, the localized mirror functor $\cF^{(\bL,\bb)}: \CF^*((\bL,\bb),U)  \to \Hom_{\cA_\bL}(\cF^{(\bL,\bb)}(\bL,\bb), \cF^{(\bL,\bb)}(U))$ induces a quasi-isomorphism. When restricted to the subcategory split-generated by $(\bL,\bb)$, the localized mirror functor $\cF^{(\bL,\bb)}$ induces a quasi-equivalence.
\end{lem}
\begin{proof}
	Given a compact, relatively spin, oriented Lagrangian $\bL$ of dimensional $n$, Fukaya \cite{Fuk10} showed that the $A_\infty$-algebra $\CF^*(\bL,\bL)$ has a cyclic structure. In other words, $$\langle m_k(v_1,\ldots, v_k), v_{k+1} \rangle=  (-1)^{\mid v_1 \mid'(\mid v_2 \mid'+ \ldots +\mid v_{k+1} \mid')} \langle m_k(v_2,\ldots, v_{k+1}), v_{1} \rangle,$$ where $\langle \cdot ,\cdot \rangle$ is a graded commutative intersection paring on $\CF^*(\bL,\bL)$, $|v_1|'= \mathrm{deg}(v)-1$ and any $v_1,\ldots, v_{k+1}\in C^*(\bL \times_\iota \bL)$. Therefore, $\cF^{(\bL,\bb)}(\bL,\bb)$ gives a finite-length, self-dual complex of $\cA_\bL$-bimodule up to shift. If it is a projective resolution of $\cA_\bL$ as an $\cA_\bL$-bimodule, then $\cA_\bL$ is a $n-$Calabi-Yau algebra.
	
	For the second statement, one can argue as in the proof of Theorem \ref{thm:qiso}. The claim follows from the Spectral sequence's computation. In particular, if we restrict to $\cF^{(\bL,\bb)}$ the subcategory split generated by $(\bL,\bb)$, the localized mirror functor induces a quasi-equivalence $\cF^{(\bL,\bb)}:\Fuk^{sub}(M) \to \mathrm{Perf}(\cA_\bL)$.
\end{proof}

\begin{rem}
	If the formal deformation space $\cA_\bL$ is the coordinate ring of a nonsingular affine variety, the projective bimodule resolution $\cF^{(\bL,\bb)}(\bL,\bb)$ of $\cA_\bL$ provides the resolution of the diagonal. In particular, this occurs when $D$ is the affine $A_0$ graph. In this case, $\bL$ is the pinched torus, and $\cA_\bL \cong \C[x,y],$ see \cite[Lem. 3.6]{HKL23}.
	
	Moreover, this result generalizes Theorem 4.9 of \cite{CHL21}, 
	which shows that the localized mirror functor possesses a certain injectivity property.
\end{rem}

Recall that in \cite{Gin07}, Ginzburg proved that a quiver with superpotential is Calabi-Yau of dimension 3 if and only if the extended cotangent complex of the path algebra, a complex of bimodules, is a resolution. This algebraic criterion admits a mirror counterpart on the symplectic side.

Indeed, for a compact Lagrangian immersion $\bL$ of dimension 3, we observe that $\cF^{(\bL,\bb)}(\bL,\bb)$ coincides with the extended cotangent complex for the noncommutative deformation space $\cA_\bL$. To prove this, let's first define the spacetime superpotential of $\bL$, see Definition 3.6 in \cite{CHL21}, using the cyclic structure \cite{Fuk10}:
\begin{defn}
	The spacetime superpotential of $(\bL,\bb)$ is defined as
	$$\Phi_\bL:= \sum_k \frac{1}{k+1} \langle m_k(\bb,\ldots,\bb),\bb \rangle.$$
\end{defn}

The following proposition is the three-dimensional counterpart of Theorem \ref{thm:qiso}, extending the local Homological Mirror Symmetry equivalence to the setting in which the noncommutative deformation algebra $\cA_\bL$ is $3$-Calabi–Yau.
\begin{prop}[Three-dimensional analogue of Theorem \ref{thm:qiso}] \label{prop: 3d}
	Let $\bL$ be a compact, relatively spin, oriented Lagrangian immersion of dimension 3 and $\cA_\bL$ be its noncommutative deformation space. Then the following are equivalent:
	\begin{enumerate}
		\item The localized mirror functor $\cF^{(\bL,\bb)}(\bL,\bb)$
		provides a projective bimodule resolution of $\cA_\bL$.
		\item The algebra $\cA_\bL$ is $3$--Calabi--Yau.
	\end{enumerate}
\end{prop}
\begin{proof}
	As shown in \cite{CHL21}, the noncommutative deformation space $\cA_\bL$ of $\bL$ is a quiver path algebra with spacetime superpotential $\Phi_\bL$. It suffices to show $\cF^{(\bL,\bb)}(\bL,\bb)$ coincides with the extended cotangent complex of $\cA_\bL$. The remaining statements follow from Theorem 4.3 of \cite{Boc08} or Proposition 5.1.9 and Theorem 5.3.1 in \cite{Gin07}.
	
	By definition, $\cF^{(\bL,\bb)}(\bL,\bb)= (\oplus_Y \cA_\bL \otimes_{\mathbb{K}} \langle Y \rangle \otimes_{\mathbb{K}} \cA_\bL,m_1^{\bb,\bb}(\cdot))$, which is of the following form: 
	\begin{equation}
		\bigoplus_{M_i} \cA_\bL e_i  \langle M_i \rangle  e_i \cA_\bL \xrightarrow{d_1} \bigoplus_{X_a} \cA_\bL e_{t(X_a)}  \langle X_a \rangle  e_{h(X_a)} \cA_\bL \xrightarrow{d_2} \bigoplus_{\bar{X_a}} \cA_\bL e_{h(X_a)}  \langle \bar{X_a} \rangle  e_{t(X_a)} \cA_\bL \xrightarrow{d_3} \bigoplus_{P_i} \cA_\bL e_{i}  \langle P_i \rangle  e_{i} \cA_\bL,
	\end{equation} where $M_i$ (resp. $P_i$) is the maximum (resp. minimum) point of $i$-th connected component in the normalization, $X_a$ is the degree 1 immersed sector, and $\bar{X_a}$ is its complement. 
	
	For simplicity, we will omit the sign of the differentials as before. The first differential can be computed using the fact that the maximum point is the unit element. Thus, $d_1(M_i)=m_1^{\bb,\bb}(M_i)= \sum_{h(X_a)=i} x_a X_a -\sum_{t(Y_a)=i} y_a Y_a.$ Moreover, observe that $$d_2(X_a)= \partial_{x_a} m_0^{\bb}= \partial_{x_a} \left ( \sum_{x_b} \partial_{x_b} \Phi_\bL \right ).$$ In other words, $d_2$ can be computed by the partial derivative of the obstruction equation $m_0^{\bb}$, which is the relations of the quiver with spacetime superpotential $\cA_\bL$. It remains to compute $d_3$. Notice that $m_1^{\bb}(\bar{X_a})= \sum_k m_k (\bb, \ldots, \bb, \bar{X_a},\bb,\ldots,\bb)$, which has degree 3. Thus $m_k (\bb, \ldots, \bb, \bar{X_a},\bb,\ldots,\bb)$ is a linear combination of $P_i$. Therefore, the coefficient of $P_i$ equals to $\langle m_k (\bb, \ldots, \bb, \bar{X_a},\bb,\ldots,\bb), M_i \rangle = \pm \langle m_k (M_i, \bb, \ldots, \bb, \bar{X_a},\bb,\ldots,\bb), \bb \rangle$ due to the cyclic structure. Since $M_i$ is the unit, this is zero unless $k=2$ and $i= t(X_a)$ or $i=h(X_a)$. Furthermore, $\langle m_2(\bb, \bar{X_a}), M_i \rangle = -\langle m_2(M_i,\bb), \bar{X_a} \rangle$, which is nonzero if and only if $m_2(M_i,\bb)=m_2(M_i,x_aX_a)=x_aX_a$. Therefore, $$d_3(\bar{X_a})= m_2(x_aX_a, \bar{X_a}) \pm m_2( \bar{X_a},x_a X_a)= x_a P_{t(X_a)} \pm x_a P_{h(X_a)}.$$ In summary, $\cF^{(\bL,\bb)}(\bL,\bb)$ is a complex of $\cA_\bL$-bimodules, which coincides with the extended cotangent complex. Hence, $\cF^{(\bL,\bb)}(\bL,\bb)$ gives a bimodule resolution of $\cA_\bL$ if and only if $\cA_\bL$ is a Calabi-Yau algebra. 
\end{proof}

\begin{rem}
	In particular, by Corollary \ref{cor: res}, the above proposition implies if $\cA_\bL$ is a 3-Calabi-Yau algebra, then the localized mirror functor $\cF^{(\bL,\bb)}$ induces a quasi-equivalence. Typical examples include the noncommutative crepant resolutions of conifold \cite{FHLY17} and $\C^3/\Z_3$ \cite{LNT23}.
\end{rem}

In addition, Bocklandt \cite{Boc08} introduced another definition of Calabi-Yau algebras: an algebra $A$ is a Calabi-Yau algebra if the derived category of finite-dimensional representations $D^b(\mathrm{Rep} \,A)$ is Calabi-Yau. Moreover, he classified such Calabi-Yau algebras of dimension less than and equal to 3. In these cases, the self-dual bimodule resolutions also coincide with the Lagrangian Floer complex $\cF^{(\bL,\bb)}(\bL,\bb)$.

Moreover, Keller \cite{Kel08} also showed for a minimal $A_\infty$-category $\cA$ with cyclic structure, the perfect derived category $\mathrm{perf}(\cA)$ is Calabi-Yau.
   
Motivated by the cyclic structure of the $A_\infty$-algebra, together with the existing results concerning Calabi-Yau algebras or Calabi-Yau categories, see for example \cite{Gin07,Boc08,Kel08,VdB15}, it is natural to expect the following:
\begin{Conjecture}
	Let $\bL$ be a compact, relatively spin, graded oriented immersed Lagrangian of dimension $n$ in a symplectic manifold $M$. Denote its Maurer-Cartan algebra (formed by degree-one elements) by $\cA_\bL$. Then the followings are equivalent:
	\begin{enumerate}
		\item $\cF^{(\bL,\bb)}(\bL,\bb)$ provides a projective resolution of $\cA_\bL$ as an $\cA_\bL$-bimodule.
		\item $\cA_\bL$ is an $n-$Calabi-Yau algebra.
	\end{enumerate}
\end{Conjecture}

If this conjecture holds, the localized mirror functor $\cF^{(\bL,\bb)}$ inherits fully-faithfulness properties whenever the localized mirror is a Calabi–Yau algebra.

\begin{rem}
	One may observe that Theorem \ref{thm:qiso} provides further evidence in support of the conjecture. In fact, under the definition introduced in \cite{Boc08}, the conjecture holds in dimensions less than or equal to 3, according to Theorem \ref{thm:qiso} above and Theorems 3.2 and 4.3 in \cite{Boc08}.
	
	However, not all noncommutative deformation spaces arising from Lagrangians are $n-$Calabi-Yau algebras. For example, consider the plumbing space of $T^*\bS^2$ according to the $ADE$ Dynkin diagram. Then the noncommutative deformation space of the zero section is not $2$-Calabi-Yau.
	
	In general, one shall consider the extended mirror functors \cite{CHL21, Hon23} to dg modules over a differential graded algebra. This is a part of an ongoing joint work with Hansol Hong.
\end{rem}

	\bibliography{mybib}{}
	\bibliographystyle{alpha} 
\end{document}